\documentclass[a4paper,english]{article}
\usepackage{a4wide}

\usepackage{amsthm}

\newtheorem{theorem}{Theorem}

\newtheorem{remark}[theorem]{Remark}
\newtheorem{definition}[theorem]{Definition}
\newtheorem{lemma}[theorem]{Lemma}

\usepackage{graphicx}
\usepackage{amsmath}
\usepackage{epstopdf} % needed if you have eps figures in a pdflatex manuscript
\usepackage{amssymb}

\usepackage[english]{babel}

\usepackage[bookmarks=true,hidelinks]{hyperref}

\usepackage{mathrsfs}

\usepackage{enumitem}

\usepackage{booktabs}

\usepackage{floatrow}

\usepackage{tikz}
\usepackage{pgfplots}
\pgfplotsset{width=.6\textwidth,compat=newest}

\definecolor{mycolor}{rgb}{0.122, 0.435, 0.698}

\usepackage{functan}

\Macro{Linfty0T01}{\mathrm{L}^{\infty}((0,1)\times(0,T))}
\Macro{Linfty01}{\mathrm{L}^{\infty}(0,1)}
\Macro{Linfty0T}{\mathrm{L}^{\infty}(0,T)}
\Macro{Linfty}{\mathrm{L}^\infty}
\Macro{Linfty0TW1infty01}{\mathrm{L}^\infty(0,T;\mathrm{W}^{1,\infty}(0,1))}
\Macro{Ccinfty01R}{\mathcal{C}_c^{\infty}((0,1)\times \R)}
\Macro{Cinfty010T}{\mathcal{C}^{\infty}([0,1]\times [0,T])}
\Macro{Cc10T}{\mathcal{C}_c^1([0,T])}
\Macro{C1R}{\mathcal{C}^1(\R)}
\Macro{C2R}{\mathcal{C}^2(\R)}
\Macro{L1}{\mathrm{L}^1}
\Macro{L101}{\mathrm{L}^1(0,1)}
\Macro{L10T}{\mathrm{L}^1(0,T)}
\Macro{L101times0T}{\mathrm{L}^1((0,1)\times(0,T))}
\Macro{L101L1}{\mathrm{L}^1(0,1;\mathrm{L}^1(0,T))}
\Macro{C0TL1}{\mathcal{C}([0,T];\m{L101})}
\Macro{Lip0TL1}{\mathrm{Lip}([0,T];\m{L101})}
\Macro{BV01}{\mathrm{B.V.}(0,1)}
\Macro{BV0T}{\mathrm{B.V.}(0,T)}

\newcommand{\eps}{\varepsilon}

\newcommand{\R}{\mathbb{R}}

\newcommand{\dt}{\operatorname{dt}}
\newcommand{\dx}{\operatorname{dx}}
\newcommand{\dy}{\operatorname{dy}}
\newcommand{\ds}{\operatorname{ds}}
\newcommand{\dz}{\operatorname{dz}}

\newcommand{\sign}{\operatorname{sign}}
\newcommand{\Dplust}{D_{+}^t}
\newcommand{\Dminus}{D_{-}}

\definecolor{mycyan}{HTML}{0474BE}
\title{A convergent finite difference scheme for the Ostrovsky--Hunter equation with Dirichlet boundary conditions}
\author{J. Ridder\thanks{Department of Mathematics, Pennsylvania State University, U.S.A. (jur436@psu.edu)} \and A.\,M. Ruf\thanks{Department of Mathematics, University of Oslo, Norway (adrianru@math.uio.no)\newline
The second author has received funding from the European Union's Framework Programme for Research and Innovation Horizon 2020 (2014-2020) under the Marie Sk{\l}odowska-Curie Grant Agreement No. 642768.}}
\begin{document}

% \title{A convergent finite difference scheme for the Ostrovsky--Hunter equation with Dirichlet boundary conditions\thanks{%
% % \begin{minipage}{1cm}
% % \includegraphics[height=1cm]{flag_yellow_low.eps}
% % \end{minipage}
% % \vspace{-1cm}\flushright
% % \hili{Include European flag}
% The second author has received funding from the European Union's Framework Programme for Research and Innovation Horizon 2020 (2014-2020) under the Marie Sk{\l}odowska-Curie Grant Agreement No. 642768.
% % General acknowledgments should be placed at the end of the article.
% }}
% % \subtitle{Do you have a subtitle?\\ If so, write it here}

% \titlerunning{A convergent finite difference scheme for the Ostrovsky--Hunter equation \ldots}        % if too long for running head

% \author{J. Ridder         \and
%         A. M. Ruf %etc.
% }

% %\authorrunning{Short form of author list} % if too long for running head

% \institute{J. Ridder \at
%               Department of Mathematics\\
%               Pennsylvania State University\\
%               University Park, U.S.A.\\
%               \email{jur436@psu.edu}           %  \\
% %             \emph{Present address:} of F. Author  %  if needed
%            \and
%            A. M. Ruf \at
%               Department of Mathematics\\
%               University of Oslo\\
%               Oslo, Norway\\
%               \email{adrianru@math.uio.no}
% }

% \date{Received: date / Accepted: date}
% % The correct dates will be entered by the editor

\maketitle

\begin{abstract}
We prove convergence of a finite difference scheme to the unique entropy solution of a general form of the Ostrovsky--Hunter equation on a bounded domain with non-homogeneous Dirichlet boundary conditions. Our scheme is an extension of monotone schemes for conservation laws to the equation at hand. The convergence result at the center of this article also proves existence of entropy solutions for the initial-boundary value problem for the general Ostrovsky--Hunter equation. Additionally, we show uniqueness using Kru\v{z}kov's doubling of variables technique. We also include numerical examples to confirm the convergence results and determine rates of convergence experimentally.

% %Include keywords and mathematical subject classification numbers as needed.
% \keywords{Ostrovsky--Hunter equation \and Short-pulse equation \and Vakhnenko equation \and Finite difference methods \and Monotone scheme \and Existence \and Uniqueness \and Stability \and Convergence \and Entropy solution \and Dirichlet boundary conditions}
% % \PACS{PACS code1 \and PACS code2 \and more}
% \subclass{65M06 \and 35M33 \and 35L35}
\end{abstract}

\section{Introduction}

We consider the initial-boundary value problem
\begin{subequations}
\begin{align}
    u_t+f(u)_x &=\gamma\int_0^x u(y,t)\dy,\label{OHeq PDE}\\
    u(x,0)&=u_0(x),\\
    u(0,t)&=\alpha(t),\\
    u(1,t)&=\beta(t),
\end{align}\label{OHeq}
\end{subequations}
with $f\in \m{C2R}$ and $\gamma >0$. Equation \eqref{OHeq PDE} is derived by integrating the nonlinear evolution equation
\begin{equation}
  (u_t+ f(u)_x)_x =\gamma u, \label{first equation}
\end{equation}
in space. This equation was posed by Ostrovsky \cite{ostrovsky1978nonlinear} and Hunter \cite{hunter1990numerical} with $f(u)=\frac{1}{2}u^2$ as a model for small-amplitude long waves on a shallow rotating fluid and is referred to as the Ostrovsky--Hunter equation \cite{boyd2005ostrovsky,liu2010wave,coclite2014some}, short wave equation \cite{hunter1990numerical},
Vakhnenko equation \cite{vakhnenko1992solitons,parkes1993stability,vakhnenko1998two,morrison1999n,vakhnenko2002calculation},
Ostrovsky--Vakhnenko equation \cite{BOUTETDEMONVEL2014189,brunelli2013hamiltonian}
and reduced Ostrovsky equation \cite{ostrovsky1978nonlinear,stepanyants2006stationary,parkes2007explicit}. If $f(u)=-\frac{1}{6}u^3$, equation \eqref{first equation} is known as the short pulse equation, which was introduced by Sch\"afer and Wayne \cite{schafer2004propagation} as a model for the propagation of ultra-short light pulses in silica optical fibers (see also \cite{Amiranashvili2010,liu2009wave}). In the present paper, however, we will consider an arbitrary flux $f\in\m{C2R}$ and will refer to equation \eqref{OHeq PDE} with general $f$ as Ostrovsky--Hunter equation.

In order to derive equation \eqref{OHeq PDE}, we integrate equation \eqref{first equation} in space to get
\begin{align}
\begin{split}
  u_t +f(u)_x&=\gamma P,\\
  P_x&=u.
\end{split}
\label{second equation}
\end{align}
The function $P$ must then be further specified by an additional constraint, e.g. $P(-\infty,t)=0$ (which leads to $P=\int_{-\infty}^x u $; see \cite{coclite2015oleinik}) or $\int P=0$ (implying $P=\int_{-\infty}^x u - \int_{-\infty}^{\infty} u$ on the real line or $P=\int_0^x u -\int_0^1 u$ in the unit interval; see \cite{hunter1990numerical,stepanyants2006stationary,liu2010wave,coclite2017convergent}). Here we will consider the unit interval and choose $P(0,t)=0$, which gives
\begin{equation}
  P[u](x,t)=\int_0^x u(t,y)\dy. \label{P equation}
\end{equation}

Concerning the initial and boundary data, we will assume
\begin{equation}
  u_0\in\m{BV01}\text{ and } \alpha,\beta\in\m{BV0T}.
\end{equation}

Coclite, di Ruvo and Karlsen developed a global well-posedness analysis utilizing the concept of entropy solutions defined in a distributional sense (see \eqref{entropy condition} in Definition~\ref{def: entropy solution} below) on the domains $\R\times \R^+$ and $\R^+\times\R^+$ in \cite{coclite2015oleinik,coclite2015well,coclite2015short,coclite2015wellposedness,coclite2016well,coclite2014some,di2013discontinuous} and on $[0,1]\times\R^+$ with non-homogeneous Dirichlet boundary conditions in \cite{coclite2018initial}. Their proofs are based on a vanishing viscosity regularization and a compensated compactness argument.

In this paper, we aim to show existence of entropy solutions (as defined in Definition~\ref{def: entropy solution} below) to the initial-boundary value problem \eqref{OHeq} by proving the convergence of a finite difference scheme. We will base our construction of the numerical scheme on the classical theory of monotone schemes for conservation laws and use central differences for the nonlocal source term. In order to get compactness of the scheme, we will employ Helly's theorem together with appropriate a priori bounds of the piecewise constant interpolation. Then, we will show convergence towards the entropy solution using discrete versions of the entropy conditions in the interior of the domain and at the boundary. Furthermore, we prove uniqueness of entropy solutions by showing $\m{L1}$ stability 
using Kru\v{z}kov's `doubling of variables' technique.
% with the help of a 'Kuznetsov type' lemma. \hili{Convergence rate?}

Without convergence proof, numerical methods for equation \eqref{first equation} are used in \cite{grimshaw2012reduced,hunter1990numerical,liu2010wave}, including Fourier pseudo-spectral methods and a finite difference scheme based on the Engquist--Osher scheme. So far the only rigorous numerical analysis of the Ostrovsky--Hunter equation is performed by Coclite, Ridder and Risebro \cite{coclite2017convergent}. The authors, however, consider the case of periodic boundary conditions and initial data with zero mean. The present paper directly extends these results to the setting of non-periodic boundary conditions. 
Although we follow the general strategy of \cite{coclite2017convergent}, the non-periodicity complicates matters throughout. In particular, we will present new versions of Harten's lemma and Kru\v{z}kov's `doubling of variables' technique that properly address the contributions of the boundary terms.
% 
% Finally, in \cite{coclite2017convergent}, the authors prove the convergence of a finite difference scheme to the unique entropy solution of \eqref{first equation} on a bounded domain with periodic boundary conditions, and thus also providing an existence proof for periodic entropy solutions. In this paper, we will consider non-periodic, non-homogeneous boundary conditions and do not assume that the initial datum has zero mean. Although this paper follows the general strategy of \cite{coclite2017convergent}, the non-periodicity complicates matters throughout.

We will consider entropy solutions of \eqref{OHeq} based on the following definition:
\begin{definition}[Entropy solution]\label{def: entropy solution}
  A function $u\in\m{C0TL1}\cap\m{Linfty0T01}$ is called an entropy solution of the Ostrovsky--Hunter equation \eqref{OHeq} if for all entropy pairs $(\eta,q)$, i.e. convex functions $\eta\in\m{C2R}$, and $q$ such that $q'=\eta'f'$,
  \begin{multline}
    \int_0^T\int_0^1 (\eta(u)\phi_t + q(u)\phi_x +\gamma \eta'(u)P[u]\phi)\dx\dt + \int_0^1 \eta(u_0(x))\phi(x,0)\dx \\
    -\int_0^1 \eta(u(x,T))\phi(x,T)\dx \geq0, \label{entropy condition}
  \end{multline}
  for all nonnegative $\phi\in\m{Ccinfty01R}$, and
  \begin{multline}
    q(u_0^{\tau}(t)) - q(\alpha(t)) -\eta'(\alpha(t))(f(u_0^{\tau}(t))-f(\alpha(t))) \leq 0 \\
    \leq q(u_1^{\tau}(t)) - q(\beta(t)) -\eta'(\beta(t))(f(u_1^{\tau}(t))-f(\beta(t)))\label{entropy bc}
  \end{multline}
  holds for a. e. $t\in(0,T)$. Here $P[u]$ is as in \eqref{P equation} and $u_0^{\tau}$ and $u_1^{\tau}$ denote the strong traces of $u$ at the boundary $x=0$ respectively $x=1$.
\end{definition}
\begin{remark}
Note that by an approximation argument, cf. \cite[pp. 57-58]{holden2015front}, a function $u\in\m{C0TL1}$ is an entropy solution if and only if inequalities \eqref{entropy condition} and \eqref{entropy bc} hold for all Kru\v{z}kov entropy pairs,
    \begin{equation*}
      \eta(u,k)=|u-k|,\qquad q(u,k)=\sign(u-k)(f(u)-f(k)),\qquad k\in\R.
     \end{equation*}
\end{remark}
\begin{remark}
This is the usual definition of entropy solutions of equation \eqref{OHeq}. However, regarding the entropy boundary condition instead of working with the original condition due to Bardos, le Roux and N\'ed\'elec \cite{doi:10.1080/03605307908820117}, we will use the entropy boundary condition \eqref{entropy bc} introduced by Dubois and LeFloch \cite{Dubois1989}.
Due to the regularizing effect of the $P$ equation \eqref{P equation} we have that $u\in\m{Linfty0T01}$ implies $P[u]\in\m{Linfty0TW1infty01}$. Therefore, if $u\in\m{Linfty0T01}$ satisfies the entropy condition \eqref{entropy condition}, then \cite[Theorem 1.1]{COCLITE20093823} assures the existence of strong traces $u^{\tau}_0, u^{\tau}_1$ and hence boundary entropy condition \eqref{entropy bc} is well-defined.
\end{remark}

% In the following we will only consider $\gamma>0$ since the other case $\gamma<0$ is covered by the reflection $x\to -x$ and $u\to -u$ of the solutions for $\gamma>0$.
% Our main theorem reads as follows.
% \begin{theorem}[Main Theorem]\label{main theorem}
%   Let $f\in\m{C2R}$, $u_0\in \m{BV01}$ and $\alpha,\beta\in\m{BV01}$. Then there exists a unique entropy solution of the Ostrovsky--Hunter equation \eqref{OHeq} which is the limit of a sequence of piecewise constant approximations obtained by a monotone finite difference scheme.
% \end{theorem}

The paper is organized as follows. In Section \ref{sec: numerical scheme} we specify the numerical scheme under consideration. Section \ref{sec: a priori estimates} contains discrete a priori bounds which are used to show compactness of the scheme.
In the next section we will develop discrete entropy inequalities both in the interior and at the boundary which will lead to our first main result, the convergence of the numerical solutions to an entropy solution, see Theorem \ref{thm: conv towards entropy solution} in Section \ref{sec: convergence towards entropy solution}.
% In the next section we will develop discrete entropy inequalities both in the interior and at the boundary and use those to show that the limit of the approximate solutions is an entropy solution.
Our second main result, the $\m{L1}$ stability and thus uniqueness of entropy solutions, is shown in Section \ref{sec: uniqueness}, Theorem \ref{thm: uniqueness}, using Kru\v{z}kov’s `doubling of variables' technique.
% In Section \ref{sec: uniqueness} we show $\m{L1}$ stability and thus uniqueness of entropy solutions using Kru\v{z}kov's `doubling of variables' technique concluding the proof of Theorem \ref{main theorem}.
%which will also be used in Section \ref{sec: convergence rates} \hili{to derive the rate of convergence of the scheme}. 
Finally, the last section provides some numerical experiments.

\section{The numerical scheme}\label{sec: numerical scheme}

We discretize the domain $[0,1]\times[0,T]$ using $(N+1)\cdot(M+2)$ grid points with $\Delta x = 1/N$ and $\Delta t = \frac{T}{M+1}$, such that for $j=0,\ldots,N$ and $n=0,\ldots,M+1$,
\begin{equation*}
  u_j^n\approx u(x_j,t^n),\qquad \text{where } x_j=j\Delta x\text{ and } t^n=n\Delta t.
\end{equation*}
As a shorthand notation for the sequence $(u_j^n)_{j=0}^N$ we will write $u^n$. 
We will also frequently use the notation $I_j=[x_{j-\frac{1}{2}},x_{j+\frac{1}{2}})$ for the interval in space, $I^n=[t^n,t^{n+1})$ for the interval in time and $I_j^n = I_j\times I^n$ for the rectangle in $[0,1]\times[0,T]$. Here, we fix the convention that $x_{j+\frac{1}{2}}=(j+\frac{1}{2})\Delta x$, $j=0,\ldots,N-1$, as well as $x_{-\frac{1}{2}} = x_0=0$ and $x_{N+\frac{1}{2}}=x_N=1$.
In order to get from the sequence $u^n$ to a function on $[0,1]\times[0,T]$ we define the piecewise constant interpolation
\begin{equation*}
  u_{\Delta t}(x,t) = u_j^n,\qquad \text{for } (x,t)\in I_j^n.
\end{equation*}
% \begin{mdframed}
The discrete initial datum $u^0$ is constructed from $u_0\in\m{BV01}$ via
\begin{equation*}
  u_j^0 = \frac{1}{\Delta x} \int_{I_j}u_0(x)\dx,\quad \text{for }j=0,\ldots,N.
\end{equation*}
Then, the numerical scheme we want to employ reads as follows: For $n\geq 0$ we set
\begin{equation}
  \begin{cases}
    u_0^{n+1} = \frac{1}{\Delta t}\int_{t^n}^{t^{n+1}}\alpha(s)\ds, &\\
    u_N^{n+1} = \frac{1}{\Delta t}\int_{t^n}^{t^{n+1}}\beta(s)\ds, &\\
    u_j^{n+1} = u_j^n - \lambda \left(F_{j+\frac{1}{2}}^n - F_{j-\frac{1}{2}}^n\right) + \gamma\Delta t P_j^n & \text{if }j=1,\ldots, N-1,
  \end{cases} \label{numscheme}
\end{equation}
% For $j=0$ and $n\in\N$, set $u_0^n=\frac{1}{\Delta t}\int_{t^n}^{t^{n+1}}\alpha(s)\ds$.
% For $j=N$ and $n\in\N$, set $u_N^n=\frac{1}{\Delta t}\int_{t^n}^{t^{n+1}}\beta(s)\ds$.
% For $j=1,2,\ldots,N-1$ and $n=0$, set $u_j^0=\frac{1}{\Delta x}\int_{x_{j-\frac{1}{2}}}^{x_{j+\frac{1}{2}}}u_0(y)\dy$.
% For $j=1,2, \ldots, N-1$ and $n\geq 0$, set
% \begin{equation}
%   u_j^{n+1}=u_j^n - \lambda \left(F_{j+\frac{1}{2}}^n - F_{j-\frac{1}{2}}^n\right) + \Delta t \gamma P_{j}^n,\label{numscheme}
%  \end{equation} 
where $P_j^n$ is the following approximation to the integral of $u$,
\begin{equation*}
  P_{j}^n = \Delta x \left(\frac{1}{2}u_0^n +\sum_{i=1}^{j-1} u_i^n + \frac{1}{2}u_j^n\right),
\end{equation*}
and the flux at $(x_{j+\frac{1}{2}},t^n)$ is approximated by
\begin{equation}
  F_{j+\frac{1}{2}}^n = F(u_j^n,u_{j+1}^n), \label{numerical flux function}
\end{equation}
where the discrete flux $F$ is a Lipschitz continuous function in two variables. We will assume that $F$ can be written in the form
\begin{equation*}
  F(u,v) = F_1(u)+F_2(v),
\end{equation*}
where $F_1,F_2\in\m{C1R}$, and that $F$ is consistent with $f$ and monotone in the sense that
\begin{equation}
  F(u,u) = f(u)\quad \text{and}\quad F'_1\geq 0,\,F'_2\leq 0. \label{properties of discrete flux I}
\end{equation}
Furthermore, we will assume
\begin{equation}
  \max_u \lambda(F'_1(u) - F'_2(u))\leq 1,\label{properties of discrete flux II}
\end{equation}
where $\lambda = \frac{\Delta t}{\Delta x}$.
% \end{mdframed}
Two examples for discrete flux functions with the assumed properties are the Lax-Friedrichs flux, i.e.
\begin{equation*}
  F_1(u) = \frac{1}{2}f(u) + \frac{1}{2\lambda}u,\qquad F_2(v) = \frac{1}{2}f(v) - \frac{1}{2\lambda}v,
\end{equation*}
and the Engquist-Osher flux, i.e.
\begin{equation*}
  F_1(u)=\int_0^u\max(f'(z),0)\dz + f(0),\qquad F_2(v) = \int_0^v\min(f'(z),0)\dz + f(0),
\end{equation*}
which satisfy \eqref{properties of discrete flux I} and \eqref{properties of discrete flux II} provided that the grid satisfies the CFL condition
\begin{equation*}
  \max_{u}|f'(u)|\lambda \leq 1.
\end{equation*}
Note that, using our scheme, we can recover a discrete version of \eqref{second equation}, since
\begin{align*}
  \Dplust u_j^n + \Dminus F_{j+\frac{1}{2}}^n &= \gamma P_j^n\\
  \Dminus P_j^n &= \frac{1}{2}\left(u_j^n+u_{j-1}^n\right),
\end{align*}
were we used the following difference operators:
\begin{equation*}
  \Dplust a^n = \frac{1}{\Delta t}\left(a^{n+1}-a^n\right)\qquad \text{and}\qquad \Dminus a_j = \frac{1}{\Delta x}\left(a_j-a_{j-1}\right).
\end{equation*}

\section{Discrete a priori estimates}\label{sec: a priori estimates}

In this section we aim to prove compactness of the scheme using Helly's theorem. This requires an $\m{Linfty}$ bound, a BV bound and a bound on the discrete time derivative of the numerical solution.
These bounds are similar to the ones in \cite{coclite2017convergent}, but the boundary conditions lead to additional terms.

\begin{lemma}[$\m{Linfty}$ bound]\label{Linfty bound}
  For $n\Delta t \leq T$, the solution $u^n$ of the numerical scheme \eqref{numscheme} satisfies
  \begin{equation*}
    \norm*{\infty}{u^n}\leq e^{\gamma T}\left(\norm*{\infty}{u^0} + \norm*{\infty}{\alpha}+\norm*{\infty}{\beta}\right).
  \end{equation*}
\end{lemma}

\begin{proof}
  For $j=1,\ldots,N-1$ we define $v_j^n = \norm*{\infty}{u^n}$. Then $v_j^n\geq u_j^n$ for all $j=1,\ldots,N-1$ and thus, by monotonicity and consistency of the scheme \eqref{properties of discrete flux I}-\eqref{properties of discrete flux II},
  \begin{align*}
    u_j^n -\lambda (F(u_j^n,u_{j+1}^n)- & F(u_{j-1}^n,u_j^n)) \\
    &= u_j^n -\lambda (F_1(u_j^n)-F_2(u_j^n)) + \lambda F_1(u_{j-1}^n) - \lambda F_2(u_{j+1}^n)\\
    &\leq v_j^n -\lambda(F_1(v_j^n)-F_2(v_j^n)) + \lambda F_1(v_{j-1}^n)-\lambda F_2(v_{j+1}^n)\\
    &= v_j^n -\lambda (F(v_j^n,v_{j+1}^n)-F(v_{j-1}^n,v_j^n))\\
    &= v_j^n.
  \end{align*}
  Hence, we have
  \begin{equation*}
    |u_j^{n+1}|\leq \norm*{\infty}{u^n}+\gamma\Delta t|P_j^n|.
  \end{equation*}
  for $j=1,2,\ldots,N-1$. Because $N\Delta x=1$, also $P_j^n$ is bounded:
  \begin{equation}
    \left|P_j^n\right| \leq \Delta x \sum_{i=0}^j |u_i^n| \leq N\Delta x \norm*{\infty}{u^n} = \norm*{\infty}{u^n}. \label{P bound}
  \end{equation}
  Regarding the boundary terms, clearly
  \begin{align*}
    &|u_0^{n+1}|\leq \norm*{\infty}{\alpha} &&\text{as well as} && |u_N^{n+1}|\leq \norm*{\infty}{\beta}.
  \end{align*}
  Thus, we have
  \begin{align*}
    \norm*{\infty}{u^n} &\leq (1+\gamma \Delta t)^n \left(\norm*{\infty}{u^0} + \norm*{\infty}{\alpha}+\norm*{\infty}{\beta}\right)\\
    &\leq e^{\gamma n\Delta t}\left(\norm*{\infty}{u^0}+ \norm*{\infty}{\alpha}+\norm*{\infty}{\beta}\right) \\
    &\leq e^{\gamma T}\left(\norm*{\infty}{u^0}+ \norm*{\infty}{\alpha}+\norm*{\infty}{\beta}\right)
  \end{align*}
  for $n\Delta t\leq T$.
\end{proof} 

The next lemma is a version of Harten's lemma \cite{harten1983high} on bounded domains that additionally uses the $\m{Linfty}$ bound from Lemma \ref{Linfty bound} to estimate the contribution of the source term to the total variation.
\begin{lemma}[B.V. bound]\label{BV bound}
  For $n\Delta t\leq T$, the solution $u^n$ of the numerical scheme \eqref{numscheme} satisfies
  \begin{equation*}
    |u^{n}|_{\m{BV01}} \leq C_T \left(|u^0|_{\m{BV01}} + |\alpha|_{\m{BV01}} + |\beta|_{\m{BV01}} + \norm*{\infty}{u^0} + \norm*{\infty}{\alpha} + \norm*{\infty}{\beta}\right)
  \end{equation*}
  where $C_T$ denotes a constant depending on $\gamma$ and $T$.
\end{lemma}
\begin{proof}
  For $n=0,\ldots,M$, we have
  \begin{align}
    |u^{n+1}|_{\m{BV01}} &= \sum_{j=0}^{N-1}|u_{j+1}^{n+1}-u_{j}^{n+1}|\notag\\
    &=|u_1^{n+1}-u_0^{n+1}| + \sum_{j=1}^{N-2}|u_{j+1}^{n+1}-u_{j}^{n+1}| + |u_N^{n+1}-u_{N-1}^{n+1}|.\label{BV bound first step}
  \end{align}
The scheme \eqref{numscheme} can then be written in conservative form, i.e. for $j=1,\ldots, N-1$ we have
\begin{equation*}
  u_j^{n+1} = u_j^n + C_{j+\frac{1}{2}}^n(u_{j+1}^n -u_j^n) - D_{j-\frac{1}{2}}^n (u_j^n-u_{j-1}^n) +\gamma \Delta t P_j^n,
\end{equation*}
where
\begin{align*}
  C_{j+\frac{1}{2}}^n &= \lambda\frac{f(u_j^n) - F_{j+\frac{1}{2}}^n}{u_{j+1}^n -u_j^n} & \text{for }j&=1,\ldots,N-1\\
  D_{j+\frac{1}{2}}^n &= \lambda\frac{f(u_{j+1}^n) - F_{j+\frac{1}{2}}^n}{u_{j+1}^n -u_j^n} &\text{for }j&=0,\ldots,N-2.
\end{align*}
Using the consistency of the numerical flux and the mean value theorem, we get
\begin{align*}
	C_{j+\frac{1}{2}}^n &= \lambda \frac{F(u_j^n,u_j^n) - F(u_j^n,u_{j+1}^n)}{u_{j+1}^n - u_j^n}\\
	&= -\lambda F'_2(\xi) \geq 0
	\intertext{and similarly}
	D_{j+\frac{1}{2}}^n &= \lambda \frac{F(u_{j+1}^n,u_{j+1}^n) - F(u_j^n,u_{j+1}^n)}{u_{j+1}^n - u_j^n}\\
	&= \lambda F'_1(\zeta)\geq 0
\end{align*}
for all $j$ in $\{1,\ldots,N-1\}$ and $\{0,\ldots,N-2\}$ respectively. Furthermore, using the mean value theorem on the difference $F_1-F_2$ a similar calculation shows that the CFL condition \eqref{properties of discrete flux II} assures $C_{j+\frac{1}{2}}^n+D_{j+\frac{1}{2}}^n \leq 1$ for $j=1,\ldots,N-2$.
% \deleted[id=A]{The CFL condition \eqref{properties of discrete flux II} assures $0\leq C_{j+\frac{1}{2}}^n\leq 1$ for $j=1,\ldots,N-1$, $0\leq D_{j+\frac{1}{2}}^n\leq 1$ for $j=0,\ldots,N-2$ and $C_{j+\frac{1}{2}}^n+D_{j+\frac{1}{2}}^n\leq 1$ for $j=1,\ldots,N-2$.}
Now, regarding the sum on the right hand side of \eqref{BV bound first step}, we can estimate
\begin{multline*}
|u_{j+1}^{n+1}-u_j^{n+1}| \leq\\
\left|u_{j+1}^n -u_j^n +C_{j+\frac{3}{2}}^n (u_{j+2}^n-u_{j+1}^n) -(D_{j+\frac{1}{2}}^n+C_{j+\frac{1}{2}}^n)(u_{j+1}^n-u_j^n) + D_{j-\frac{1}{2}}^n(u_j^n-u_{j-1}^n)\right|\\
  +\gamma\Delta t |P_{j+1}^n-P_j^n|
  % \sum_{j=1}^{N-2} \left|u_{j+1}^n -u_j^n +C_{j+\frac{3}{2}}^n (u_{j+2}^n-u_{j+1}^n) -D_{j+\frac{1}{2}}^n(u_{j+1}^n-u_j^n)-C_{j+\frac{1}{2}}^n(u_{j+1}^n-u_j^n) + D_{j-\frac{1}{2}}^n(u_j^n-u_{j-1}^n)\right| \\
  % +\sum_{j=1}^{N-2}\gamma\Delta t |P_{j+1}^n-P_j^n|
\end{multline*}
Regarding the first sum, we get
\begin{align*}
  &\sum_{j=1}^{N-2}\left|u_{j+1}^n -u_j^n +C_{j+\frac{3}{2}}^n (u_{j+2}^n-u_{j+1}^n) -(D_{j+\frac{1}{2}}^n+C_{j+\frac{1}{2}}^n)(u_{j+1}^n-u_j^n) + D_{j-\frac{1}{2}}^n(u_j^n-u_{j-1}^n)\right|\\
&\leq \sum_{j=1}^{N-2} C_{j+\frac{3}{2}}^n|u_{j+2}^n-u_{j+1}^n| + \sum_{j=1}^{N-2}(1-D_{j+\frac{1}{2}}^n-C_{j+\frac{1}{2}}^n)|u_{j+1}^n-u_j^n| + \sum_{j=1}^{N-2}D_{j-\frac{1}{2}}^n|u_j^n-u_{j-1}^n|\\
&= \sum_{j=2}^{N-1} C_{j+\frac{1}{2}}^n|u_{j+1}^n-u_{j}^n| + \sum_{j=1}^{N-2}(1-D_{j+\frac{1}{2}}^n-C_{j+\frac{1}{2}}^n)|u_{j+1}^n-u_j^n| + \sum_{j=0}^{N-3}D_{j+\frac{1}{2}}^n|u_{j+1}^n-u_{j}^n|\\
&= \sum_{j=1}^{N-2}|u_{j+1}^n-u_j^n| -C_{\frac{3}{2}}^n|u_2^n-u_1^n| + C_{N-\frac{1}{2}}^n|u_N^n-u_{N-1}^n| -D_{N-\frac{3}{2}}^n|u_{N-1}^n-u_{N-2}^n| + D_{\frac{1}{2}}^n|u_1^n-u_0^n|
\end{align*}
On the other hand, regarding the boundary terms in \eqref{BV bound first step}, since $D_{\frac{1}{2}}^n\leq 1$ and \eqref{P bound}, we find
\begin{align*}
  |u_1^{n+1}- & u_0^{n+1}| \\
  &\leq |u_1^n - u_0^{n+1}-\lambda (F_{\frac{3}{2}}-F_{\frac{1}{2}})| +\gamma \Delta t|P_1^n|\\
  &\leq |u_0^{n+1}-u_0^n| + |u_1^n-u_0^n - \lambda(F_{\frac{3}{2}}-f(u_1^n)-(F_{\frac{1}{2}}-f(u_1^n)))| + \gamma \Delta t \norm*{\infty}{u^n}\\
  &= |u_0^{n+1}-u_0^n| + |u_1^n-u_0^n + C_{\frac{3}{2}}^n(u_2^n-u_1^n)-D_{\frac{1}{2}}^n(u_1^n-u_0^n)|+ \gamma \Delta t \norm*{\infty}{u^n}\\
  &\leq |u_0^{n+1}-u_0^n| + C_{\frac{3}{2}}^n|u_2^n-u_1^n| + (1-D_{\frac{1}{2}}^n)|u_1^n-u_0^n|+ \gamma \Delta t \norm*{\infty}{u^n}
  \intertext{and similarly}
  |u_N^{n+1}- & u_{N-1}^{n+1}| \\
  &\leq |u^{n+1}_N-u^n_N| +(1-C_{N-\frac{1}{2}}^n)|u_N^n-u_{N-1}^n| + D_{N-\frac{3}{2}}^n |u^n_{{N-1}-u^n_{N-2}}|  + \gamma \Delta t \norm*{\infty}{u^n}
\end{align*}
% and a similar estimate for $|u_N^{n+1}-u_{N-1}^{n+1}|$.
Moreover, we will estimate the $P$ term with the help of Lemma \eqref{Linfty bound} as follows
  \begin{align*}
    \gamma \Delta t \sum_{j=1}^{N-2}|P_{j+1}^n-P_j^n| &= \gamma \Delta t \frac{\Delta x}{2}\sum_{j=1}^{N-2} |u_{j+1}^n+u_j^n|\\
    &\leq \gamma \Delta t (\Delta x N) \norm*{\infty}{u^n}\\
    &\leq\gamma \Delta t e^{\gamma T}\left(\norm*{\infty}{u^0}+\norm*{\infty}{\alpha}+\norm*{\infty}{\beta}\right).
  \end{align*}
In summary we get
\begin{align*}
  |u^{n+1} & |_{\m{BV01}} \\
  &\leq |u^n|_{\m{BV01}} + |u_0^{n+1}-u_0^n| + |u_N^{n+1}-u_N^n| + \gamma \Delta t e^{\gamma T}\left(\norm*{\infty}{u^0}+\norm*{\infty}{\alpha}+\norm*{\infty}{\beta}\right).
\end{align*}
Furthermore, we note that
\begin{equation*}
  \sum_{n=0}^M |u_0^{n+1}-u_0^n| \leq |\alpha|_{\m{BV0T}},
\end{equation*}
and similarly for the right boundary.
Therefore we get
\begin{equation*}
  |u^{n}|_{\m{BV01}} \leq C_T \left(|u^0|_{\m{BV01}} + |\alpha|_{\m{BV01}} + |\beta|_{\m{BV01}} + \norm*{\infty}{u^0} + \norm*{\infty}{\alpha} + \norm*{\infty}{\beta}\right).
\end{equation*}
\end{proof}
Lastly, we have a bound on the discrete time derivative of $u_{\Delta t}$.
\begin{lemma}[Bound of the time derivative]\label{time derivative bound}
  For $n\Delta t\leq T$, the solution of the numerical scheme \eqref{numscheme} satisfies
  \begin{equation*}
    \Delta x \sum_{j=0}^{N}\left| D_+^t u_j^n\right| \leq C_{\lambda}\left(|u^0|_{\m{BV01}}+\norm*{\infty}{u^0}+\norm*{\infty}{\alpha}+\norm*{\infty}{\beta}\right),
  \end{equation*}
  where $C_{\lambda}$ depends on $\gamma, T$, the Lipschitz constant of the discrete flux $F$ and $\lambda$.
\end{lemma}

\begin{proof}
  Using the definition of the numerical scheme \eqref{numscheme}, the Lipschitz continuity of $F$, the $\m{Linfty}$ bound for $P$ as seen in \eqref{P bound}, and the BV and $\m{Linfty}$ bounds of $u^n$ from Lemma \ref{BV bound} and \ref{Linfty bound}, we get
  \begin{align*}
    \Delta x & \sum_{j=0}^{N} |D_+^t u_j^n| = \Delta x \sum_{j=1}^{N-1}|D_+^tu_j^n| + \Delta x|D_+^t u_0^n| + \Delta x |D_+^t u_N^n|\\
    &\leq \Delta x \sum_{j=1}^{N-1} |D_{-}F(u_j^n,u_{j+1}^n)| + \gamma \Delta x \sum_{j=1}^{N-1} |P_j^n| + \frac{1}{\lambda}\left(|\alpha|_{\m{BV01}}+|\beta|_{\m{BV01}}\right)\\
    &\leq C\Delta x \sum_{j=1}^{N-1}\left(|D_{-}u_j^n|+|D_{-}u_{j+1}^n|\right) + \gamma \Delta x N \norm*{\infty}{u^n}+ \frac{1}{\lambda}\left(|\alpha|_{\m{BV01}}+|\beta|_{\m{BV01}}\right)\\
    &\leq C\left(|u^n|_{\m{BV01}} +\norm*{\infty}{u^n}\right)+ \frac{1}{\lambda}\left(|\alpha|_{\m{BV01}}+|\beta|_{\m{BV01}}\right)\\
    &\leq C_{\lambda}\left(|u^0|_{\m{BV01}} + |\alpha|_{\m{BV01}} + |\beta|_{\m{BV01}} + \norm*{\infty}{u^0} + \norm*{\infty}{\alpha} + \norm*{\infty}{\beta}\right)
  \end{align*}
\end{proof}

With the help of these three bounds we finally can apply a version of Helly's theorem to show compactness of the scheme.

\begin{lemma}[Convergence]\label{lem: convergence}
  Let $u_{\Delta t}$ be the family of solutions of the numerical scheme \eqref{numscheme} defined by $u_{\Delta t}(x,t)=u_j^n$ for $(x,t)\in [x_{j-\frac{1}{2}},x_{j+\frac{1}{2}})\times [t^n,t^{n+1})$. Further, let $\lambda=\frac{\Delta t}{\Delta x}$ be fixed such that the discrete flux satisfies \eqref{properties of discrete flux I} and \eqref{properties of discrete flux II}. Then there is a sequence $\Delta t_k$ and a function $u\in\m{Lip0TL1}$ such that $\Delta t_k\to 0$ and $u_{\Delta t_k}$ converges to $u$ in $\m{C0TL1}$.
\end{lemma}

\begin{proof}
  We want to apply Helly's theorem \cite[Theorem A.11]{holden2015front}. This requires an $\m{Linfty}$ bound, a bound on the variation in space that is independent of $\Delta t$, and $\m{L1}$ continuity in time as $\Delta t\to 0$.
  An application of Lemma \ref{Linfty bound} gives
  \begin{align*}
    \norm{Linfty01}{u_{\Delta t}(\cdot, t)}&\leq e^{\gamma T}\left(\norm{Linfty01}{u^0} + \norm*{\infty}{\alpha}+\norm*{\infty}{\beta}\right) \\
    &\leq e^{\gamma T}\left(\norm{Linfty01}{u_0}+\norm*{\infty}{\alpha}+\norm*{\infty}{\beta}\right).
  \end{align*}
  Furthermore, by using Lemma \ref{BV bound}, we find
  \begin{align*}
    \norm{L101}{u_{\Delta t}(\cdot+\eps,t)-u_{\Delta t}(\cdot,t)} &\leq \eps |u_{\Delta t}(\cdot,t)|_{\m{BV01}}\\
    &\leq \eps \left(|u^0|_{\m{BV01}}+C\left(\norm*{\infty}{u^0}+\norm*{\infty}{\alpha}+\norm*{\infty}{\beta}\right)\right)\\
    &\leq \eps \left(|u_0|_{\m{BV01}}+C\left(\norm*{\infty}{u_0}+\norm*{\infty}{\alpha}+\norm*{\infty}{\beta}\right)\right)\\
    &\to 0,\qquad\text{as }\eps\to 0 \text{ uniformly in }\Delta t.
  \end{align*}
  Finally, in order to show continuity in time, we employ Lemma \ref{time derivative bound}. For $t\in[t^n,t^{n+1})$ and $s\in[t^{\overline{n}},t^{\overline{n}+1})$ with $n>\overline{n}$ we find
  \begin{align*}
    \int_0^1 |u_{\Delta t}(x,t)-u_{\Delta t}(x,s)|\dx &= \Delta x \sum_{j=0}^{N}|u_j^n-u_j^{\overline{n}}|\\
    &\leq \Delta x \sum_{l=\overline{n}}^{n-1}\sum_{j=0}^N |u_j^{l+1}-u_j^l|\\
    & = \Delta t \sum_{l=\overline{n}}^{n-1}\Delta x\sum_{j=0}^N |D_{+}^t u_j^l|\\
    &\leq \Delta t (n-\overline{n})C_{\lambda} \left(|u^0|_{\m{BV01}}+\norm*{\infty}{u^0}+\norm*{\infty}{\alpha}+\norm*{\infty}{\beta}\right)\\
    &=(t^n-t^{\overline{n}}) C_{\lambda}\left(|u^0|_{\m{BV01}}+\norm*{\infty}{u^0}+\norm*{\infty}{\alpha}+\norm*{\infty}{\beta}\right)\\
    &\leq C_{\lambda}|t-s| + \mathcal{O}(\Delta t).
  \end{align*}
  An application of Helly's theorem assures the existence of a sequence $\Delta t_k\to 0$  and a function $u\in\m{Lip0TL1}$ such that such that $u_{\Delta t_k}$ converges to $u$ in the space $\m{C0TL1}$ as $k\to\infty$.
\end{proof}

\section{Convergence towards the entropy solution}\label{sec: convergence towards entropy solution}

In this section we prove that the numerical scheme converges to an entropy solution of the Ostrovsky--Hunter equation. This fact hinges on discrete entropy inequalities for the interior of the domain and the boundary. These inequalities require a discrete version of the entropy flux that is consistent with the numerical flux function \eqref{numerical flux function}:
\begin{equation}
  Q_{j+\frac{1}{2}}^n = Q(u_j^n,u_{j+1}^n),\qquad Q_{j-\frac{1}{2}}^n = Q(u_{j-1}^n,u_j^n), \label{discrete entropy flux}
\end{equation}
where
\begin{equation*}
  Q(u,v) = \int_c^u \eta'(z)F'_1(z)\dz + \int_c^v \eta'(z)F'_2(z)\dz,
\end{equation*}
and $c\in\R$ is an arbitrary constant. Note that since $F_1$ and $F_2$ are Lipschitz continuous, and if $\eta'$ is bounded, also $Q$ is Lipschitz continuous in both variables.

We will now derive discrete versions of the entropy conditions \eqref{entropy condition} and \eqref{entropy bc}. The entropy condition in the interior of the domain has already been proven in \cite{coclite2017convergent}.

\begin{lemma}[Discrete Entropy inequalities]\label{lem: discrete entropy conditions}
  For any convex entropy $\eta\in\m{C2R}$ with entropy flux $q$ given by $q'=\eta'f'$, let $Q_{j+\frac{1}{2}}^n$ and $Q_{j-\frac{1}{2}}^n$ be defined by \eqref{discrete entropy flux}. Then the solutions of the scheme satisfies for each $n$
  \begin{equation}
    \Dplust \eta_j^n +  \Dminus Q_{j+\frac{1}{2}}^n - \gamma \eta_j^{\prime,n+1}P_j^n \leq 0 \label{discrete entropy condition}
  \end{equation}
  for $j=1,\ldots,N-1$, as well as
  \begin{equation}
    Q_{\frac{1}{2}}^n - q(u_0^n) - \eta'(u_0^n)(F_{\frac{1}{2}}^n - f(u_0^n)) \leq 0 \label{discrete entropy bc}
  \end{equation}
  and
  \begin{equation*}
    Q_{N-\frac{1}{2}}^n - q(u_N^n) - \eta'(u_N^n)(F_{N-\frac{1}{2}}^n - f(u_N^n)) \geq 0.
  \end{equation*}
\end{lemma}

\begin{proof}
  The first inequality is derived in \cite[Lemma 5]{coclite2017convergent} (see also \cite[Lemma 6.1]{karlsen2014L1}). For the second inequality we use a Taylor approximation and the convexity of the flux
  \begin{align*}
    Q_{\frac{1}{2}}^n - q(u_0^n) - & \eta'(u_0^n)(F_{\frac{1}{2}}^n - f(u_0^n)) \\
    &= Q(u_0^n,u_1^n) - Q(u_0^n,u_0^n) - \eta'(u_0^n)(F(u_0^n,u_1^n) - F(u_0^n,u_0^n))\\
    &= \int_c^{u_1^n}\eta'(z)F_2'(z)\dz -\int_c^{u_0^n}\eta'(z)F_2'(z)\dz - \eta'(u_0^n)(F_2(u_1^n) - F_2(u_0^n))\\
    &=\int_{u_0^n}^{u_1^n} \eta'(z)F'_2(z)\dz - \int_{u_0^n}^{u_1^n}\eta'(u_0^n)F'_2(z)\dz\\
    &= \int_{u_0^n}^{u_1^n}\eta''(\xi)(z-u_0^n)F'_2(z)\dz \\
    &=\sign(u_1^n-u_0^n)\int_{\min(u_0^n,u_1^n)}^{\max(u_0^n,u_1^n)} \eta''(\xi) (z-u_0^n) F'_2(z) \dz\\
    &=\int_{\min(u_0^n,u_1^n)}^{\max(u_0^n,u_1^n)} \eta''(\xi) |z-u_0^n| F'_2(z) \dz\leq 0
  \end{align*}
  The proof of the third inequality can be done analogously.
\end{proof}

Thus far, we only know that a sequence of solutions of the numerical scheme \eqref{numscheme} converges to some $u\in\m{C0TL1}$. By passing to the limit in the discrete entropy conditions of Lemma \ref{lem: discrete entropy conditions} we can now show that $u$ is in fact an entropy solution. To accomplish that we will employ similar techniques as in \cite{coclite2017convergent} in regards to the entropy condition and as in \cite{s2006weak} in regards to the entropy boundary condition.
While the following theorem only provides the convergence of a subsequence of $u_{\Delta t}$, the uniqueness result in Section \ref{sec: uniqueness} ensures that the whole sequence converges to the unique entropy solution.
\begin{theorem}[Convergence towards the entropy solution]\label{thm: conv towards entropy solution}
  Let $u_0\in\m{BV01}$ and $\alpha,\beta\in\m{BV0T}$ and fix $\lambda =\frac{\Delta t}{\Delta x}$ such that the discrete flux in the scheme defined by \eqref{numscheme} satisfies the \eqref{properties of discrete flux I} and \eqref{properties of discrete flux II}. Then for any sequence $(\Delta t_n)_n$ such that $\Delta t_n\to 0$, there is a subsequence $\Delta t_{n_k}$ such that the piecewise constant interpolations $u_{\Delta t_{n_k}}$ defined by the scheme \eqref{numscheme} converge in $\m{C0TL1}$ towards an entropy solution of the Ostrovsky--Hunter equation as $k\to\infty$.
\end{theorem}

\begin{proof}
  Let $(u_{\Delta t_{n_k}})$ be a sequence of approximate solutions that converges to $u$ in the space $\m{C0TL1}$ as $\Delta t_{n_k}\to 0$ (cf. Lemma \ref{lem: convergence}). For simplicity, we will omit any indices on $\Delta t$.
  According to Lemma \ref{lem: discrete entropy conditions}, the function $u_{\Delta t}$ satisfies the discrete entropy and entropy boundary conditions.
  
  First, we show that $u$ satisfies the entropy condition \eqref{entropy condition}. 
  Multiplying the discrete entropy condition \eqref{discrete entropy condition} by $\Delta t\Delta x \phi_j^n$, where $\phi_j^n = \frac{1}{\Delta t\Delta x}\iint_{I_j^n}\phi(x,t)\dx\dt$ for some nonnegative test function $\phi\in\m{Ccinfty01R}$, and taking the sum over $n=0,\ldots,M$ and $j=1,\ldots,N-1$ gives
  \begin{align}
    0 &\geq \Delta t\Delta x \sum_{n=0}^M\sum_{j=1}^{N-1} \left(\phi_j^n \Dplust \eta_j^n + \phi_j^n\Dminus Q_{j+\frac{1}{2}}^n -\gamma \phi_j^n \eta_j^{\prime,n+1}P_j^n \right) \notag\\
    &= \Delta x\sum_{j=1}^{N-1}(\phi_j^{M+1}\eta_j^{M+1}-\phi_j^0\eta_j^0) -\Delta t\Delta x \sum_{n=0}^M\sum_{j=1}^{N-1} \eta_j^{n+1}\Dplust\phi_j^n \label{discrete weak entropy condition}\\
    &\qquad  % +\Delta t \sum_{n=0}^M(\phi^n_N Q_{N-\frac{1}{2}}^n - \phi_0^n Q_{\frac{1}{2}}^n) 
    - \Delta t\Delta x \sum_{n=0}^M\sum_{j=1}^{N} Q_{j-\frac{1}{2}}^n \Dminus\phi_j^n
    -\gamma \Delta t\Delta x \sum_{n=0}^M\sum_{j=1}^{N-1} \phi_j^n\eta_j^{\prime,n+1}P_j^n,\notag
  \end{align}
  where we have used that $\phi_0^n=\phi_N^n=0$ for $\Delta x$ small enough.
  As in \cite{coclite2017convergent} we can pass to the limit $\Delta t\to 0$ in inequality \eqref{discrete weak entropy condition}.

  More precisely, the continuity of $\eta$ and the convergence of $u_{\Delta t}$ imply that $\eta(u_{\Delta t})$ converges to $\eta(u)$ in $\m{C0TL1}$. On the other hand, since both the numerical and continuous entropy fluxes are Lipschitz continuous and $u_{\Delta t}(\cdot,t)$ has bounded variation for all $t\in[0,T]$, we find
  \begin{align*}
    \sum_{n=0}^M\sum_{j=1}^N \iint_{I_j^n} & \left|Q_{j-\frac{1}{2}}^n-q(u(x,t))\right|\dx\dt \\
    &\leq \sum_{n=0}^M\sum_{j=1}^N \iint_{I_j^n}\left(|Q_{j-\frac{1}{2}}^n- q(u_j^n)| +|q(u_j^n)-q(u(x,t))|\right)\dx\dt \\
    &\leq C\sum_{n=0}^M\sum_{j=1}^N \iint_{I_j^n}\left(|u_{j-1}^n-u_j^n| +|u_j^n-u(x,t)|\right)\dx\dt \\
    & \leq C_T \Delta x + C \int_0^T\int_0^1 |u_{\Delta t}-u| \dx\dt \to 0.
  \end{align*}
  Finally, the $\m{L1}$ convergence of $u_{\Delta t}$ implies $\m{Linfty}$ convergence of the P term, since for $x\in I_j$ we have
  \begin{align*}
    |P_j^n - P[u](x,t)| &= \left| \Delta x\left( \sum_{i=0}^{j-1}u_i^n +\frac{1}{2}u_j^n \right) -\int_0^x u(y,t)\dy \right|\\
    &\leq \int_0^x |u_{\Delta t}(y,t)-u(y,t)|\dy +C\Delta x \norm{Linfty01}{u_{\Delta t}(\cdot,t)}\\
    &\leq \norm{L101}{u_{\Delta t}(\cdot,t)-u(\cdot,t)} +C\Delta x \norm{Linfty01}{u_{\Delta t}(\cdot,t)}\to 0.
  \end{align*}
  Thus we can pass to the limit $\Delta t\to 0$ in \eqref{discrete weak entropy condition} and get
  \begin{multline*}
    0\geq \int_0^1 \eta(u(x,T))\phi(x,T)\dx - \int_0^1\eta(u(x,0))\phi(x,0)\dx \\- \int_0^T\int_0^1 \left(\eta(u)\phi_t + q(u)\phi_x + \gamma \eta'(u)P[u] \phi \right) \dx\dt
  \end{multline*}
  and therefore $u$ satisfies the entropy condition in the interior of the domain.
  % In view of \cite[Theorem 2]{coclite2017convergent} it suffices to show that the continuous entropy boundary conditions are fulfilled.
  
  Regarding the entropy boundary condition \eqref{entropy bc},
  rearranging \eqref{discrete entropy condition} yields
  \begin{equation*}
    Q_{j+\frac{1}{2}}^n \leq Q_{j-\frac{1}{2}}^n - \Delta x \Dplust \eta_j^n + \gamma \Delta x P_j^n \eta_j^{\prime,n+1}
   \end{equation*}
   Multiplying by $\Delta t \psi^n$, where $\psi^n = \frac{1}{\Delta t}\int_{t^n}^{t^{n+1}}\psi(s)\ds$ for some nonnegative test function $\psi\in\m{Cc10T}$, and summing over $n=0,\ldots,M$, we get
   \begin{align*}
    \Delta t \sum_{n=0}^M & Q_{j+\frac{1}{2}}^n\psi^n \leq \Delta t \sum_{n=0}^M Q_{j-\frac{1}{2}}^n\psi^n - \Delta x\Delta t \sum_{n=0}^M \Dplust\eta_j^n\psi^n + \gamma \Delta x\Delta t\sum_{n=0}^M P_j^n\eta_j^{\prime,n+1} \psi^n\\
    &= \Delta t \sum_{n=0}^M Q_{j-\frac{1}{2}}^n\psi^n + \Delta x\Delta t \sum_{n=0}^M \underbrace{\eta_j^{n+1}}_{\leq \norm*{\infty}{\eta'}\norm*{\infty}{u_{\Delta t}}+C} \underbrace{\Dplust \psi^n}_{\norm*{\infty}{\psi'}} + \gamma \Delta x\Delta t\sum_{n=0}^M P_j^n\eta_j^{\prime,n+1} \psi^n\\
    &\leq \Delta t \sum_{n=0}^M Q_{j-\frac{1}{2}}^n\psi^n + CT\Delta x + \gamma \Delta x\Delta t\sum_{n=0}^M P_j^n\eta_j^{\prime,n+1} \psi^n.
   \end{align*}
   Repeating this argument and using the discrete entropy boundary condition \eqref{discrete entropy bc} yields
   \begin{align}
   \begin{split}
    \Delta t & \sum_{n=0}^M Q_{j+\frac{1}{2}}^n\psi^n \leq \Delta t\sum_{n=0}^M Q_{\frac{1}{2}}^n \psi^n + jCT\Delta x + \gamma \Delta x\Delta t \sum_{i=1}^j\sum_{n=0}^M P_i^n \eta_i^{\prime,n+1}\psi^n\\
    &\leq \Delta t \sum_{n=0}^M (q(u_0^n)+\eta'(u_0^n)(F_{\frac{1}{2}}^n -f(u_0^n)))\psi^n + jCT\Delta x + \gamma \Delta x\Delta t \sum_{i=1}^j\sum_{n=0}^M P_i^n \eta_i^{\prime,n+1}\psi^n. \label{discrete almost entropy bc}
   \end{split}
   \end{align}
   In order to recover the entropy boundary condition \eqref{entropy bc} we now pass to the limit $\Delta t\to 0$ and then $x\to 0$.

   Firstly, since $u_{\Delta t}$ converges to $u$ in $\m{C0TL1}$ and thus also in $\m{L101times0T}$, using the Lipschitz continuity of $Q$, we find
   \begin{align*}
     & \sum_{n=0}^M \sum_{j=0}^N \iint_{I_j^n} |Q_{j+\frac{1}{2}} - q(u(x,t))| \dx\dt \\
     &\leq \sum_{n=0}^M \sum_{j=0}^N \iint_{I_j^n} (|Q(u_j^n,u_{j+1}^n)-Q(u(x,t),u_{j+1}^n)| + |Q(u(x,t),u_{j+1}^n) - q(u(x,t))|) \dx\dt\\
     &\leq C \sum_{n=0}^M \sum_{j=0}^n \iint_{I_j^n} (|u_j^n -u(x,t)|+|u_{j+1}^n-u(x,t)|)\dx\dt\\
     &\leq C \sum_{n=0}^M \sum_{j=0}^n \iint_{I_j^n} (|u_{\Delta t}(x,t) -u(x,t)|+|u_{\Delta t}(x+\Delta x,t) -u_{\Delta t}(x,t)| )\dx\dt\\
     &\leq C \left(\int_0^T\int_0^1 |u_{\Delta t}(x,t)-u(x,t)|\dx\dt + T\Delta x \sup_{0\leq n\leq M+1} |u^n|_{\m{BV01}}\right) \to 0.
   \end{align*}
   Thus the left hand side of \eqref{discrete almost entropy bc} converges to $\int_0^T q(u(x,t))\psi(t)\dt$ for almost every $x\in(0,1)$.

   Because of the Lipschitz continuity of $F$ and the $\m{Linfty}$ bound in Lemma \ref{Linfty bound}, the piecewise constant interpolation in time of the values $F_{\frac{1}{2}}^n$ is bounded in $\m{Linfty0T}$. Thus there exists a subsequence such that $F_{\frac{1}{2}}^n \wkconv*{*}{} \widetilde{f}_0(t)$ in $\m{Linfty0T}$ for some $\widetilde{f}_0\in\m{Linfty0T}$.

   Since $u_0^n = \frac{1}{\Delta t} \int_{t^{n-1}}^{t^n} \alpha(s)\ds$ converges to $\alpha(t)$ for almost all $t\in(0,T)$, the continuity of $q,\eta'$ and $f$ assures convergence of the remaining terms on the right hand side of \eqref{discrete almost entropy bc}.

   Thus, by passing to the limit $\Delta t\to 0$ in \eqref{discrete almost entropy bc}, we get
   \begin{multline*}
    \int_0^T q(u(x,t))\psi(t)\dt \leq \int_0^T \left(q(\alpha(t))+\eta'(\alpha(t))\left(\widetilde{f}_0(t)-f(\alpha(t))\right)\right)\psi(t)\dt + CTx \\
    +\gamma \int_0^x\int_0^T \eta'(u)P[u]\psi(t)\dt\dx.
   \end{multline*}
   Because $u(x,\cdot)$ is of bounded variation in time, we have strong convergence in $\m{L10T}$. The limit can only be the strong trace, i.e. $u(x,\cdot)\to u_0^{\tau}$, as $x\to 0$. Thus, by passing to the limit $x\to 0$ in the foregoing inequality, we get
   \begin{equation}
    \int_0^T q(u_0^{\tau}(t))\psi(t)\dt \leq \int_0^T\left(q(\alpha(t))+\eta'(\alpha(t))\left(\widetilde{f}_0(t)-f(\alpha(t))\right)\right)\psi(t)\dt \label{almost entropy bc}
   \end{equation}
   and since $\psi\in\m{Cc10T}$ is arbitrary
   \begin{equation*}
    q(u_0^{\tau}(t))\leq q(\alpha(t)) + \eta'(\alpha(t))\left(\widetilde{f}_0(t)-f(\alpha(t))\right)
   \end{equation*}
  for almost every $t\in(0,T)$. It remains to show that $\widetilde{f}_0(t) = f(u_0^{\tau}(t))$. By an approximation argument, \eqref{almost entropy bc} also holds true for Kru\v{z}kov entropy pairs $\eta(u)=|u-k|$, $q(u)=\sign(u-k)(f(u)-f(k))$ with arbitrary $k\in\R$. Choosing $k>\max(u_0^{\tau}(t),\alpha(t))$ yields
  \begin{equation*}
    -(f(u_0^{\tau}(t))-f(k)) \leq - (f(\alpha(t))-f(k)) - \left(\widetilde{f}_0(t)-f(\alpha(t))\right)
  \end{equation*}
  and thus
  \begin{equation*}
    f(u_0^{\tau}(t)) \geq \widetilde{f}_0(t).
  \end{equation*}
   On the other hand, choosing $k<\min(u_0^{\tau}(t),\alpha(t))$ gives $f(u_0^{\tau}(t)) \leq \widetilde{f}_0(t)$, and therefore $\widetilde{f}_0(t) = f(u_0^{\tau}(t))$. This proves the entropy boundary condition at $x=0$. The boundary at $x=1$ can be handled similarly.
\end{proof}

\section{$\m{L1}$ stability and uniqueness}\label{sec: uniqueness}

We now want to prove $\m{L1}$ stability of solutions following the `doubling of variables' method introduced by Kru\v{z}kov \cite{kruvzkov1970first}. 

\begin{theorem}[$\m{L1}$ stability]\label{thm: uniqueness}
If $u$ and $v$ are entropy solutions of the Ostrovsky--Hunter equation with initial datum $u_0$ and $v_0$ respectively, then
\begin{equation*}
  \norm{L101}{u(\cdot,T)-v(\cdot,T)} \leq e^{\gamma T} \norm{L101}{u_0-v_0}.
\end{equation*}
In particular, this implies that entropy solutions to the initial-boundary value problem are unique.
\end{theorem}

\begin{proof}
  Let $u$ and $v$ be entropy solutions with initial datum $u_0$ and $v_0$ respectively. We will now consider the entropy inequality \eqref{entropy condition} with Kru\v{z}kov entropy pairs and a nonnegative test function $\phi$ with support away from $t=0$ and $t=T$. By taking \eqref{entropy condition} for $u$ in the variables $(x,t)$ and for $v$ in the variables $(y,s)$ both with the test function $\phi(x,t,y,s)$, integrating each with respect to the respective other two variables and adding them we get
  \begin{multline*}
    \int_0^T\int_0^1\int_0^T\int_0^1 \Big(|u(x,t)-v(y,s)|(\phi_t+\phi_s) + q(u(x,t),v(y,s))(\phi_x+\phi_y) \\
     +\gamma \sign(u(x,t)-v(y,s))(P[u](x,t)-P[v](y,s))\phi\Big)\dx\dt\dy\ds \geq 0 
  \end{multline*}
  Now, let $\phi = \psi(\frac{x+y}{2},\frac{t+s}{2})\omega_{\eps}(x-y)\omega_{\eps_0}(t-s)$, where $0\leq \psi\leq 1$ is a test function to be chosen later and $\omega_{\eps,\eps_0}$ are symmetric standard mollifiers. Then, using \cite[Lemma 2.9]{holden2015front}, we find that the terms not involving $P$ converge towards
  \begin{equation*}
    \int_0^T\int_0^1 \Big( |u-v|\psi_t + q(u,v)\psi_x \Big)\dx\dt,
  \end{equation*}
  as $\eps,\eps_0\to 0$. Regarding the remaining term, we use
  \begin{align*}
    |P[u](x,t) & - P[v](y,s)| \leq |P[u](x,t)-P[v](x,s)|+|P[v](x,s)-P[v](y,s)|\\
    &\leq \norm{L101}{u(\cdot,t)-v(\cdot,s)} + |x-y|\cdot \norm{Linfty01}{v(\cdot,s)}\\
    &\leq \norm{L101}{u(\cdot,t)-v(\cdot,t)} +\norm{L101}{v(\cdot,t)-v(\cdot,s)} + |x-y|\cdot \norm{Linfty01}{v(\cdot,s)}.
  \end{align*}
  Hence, using that weak solutions of bounded variation are Lipschitz continuous in time \cite[Theorem 7.10]{holden2015front}, we find
  \begin{align*}
    &\int_0^T\int_0^1\int_0^T\int_0^1 |P[u](x,t)-P[v](y,s)|\phi\dx\dt\dy\ds \\
    &\leq \int_0^T \norm{L101}{u(\cdot,t)-v(\cdot,t)}\dt  + \int_0^T\int_0^T \norm{L101}{v(\cdot,t)-v(\cdot,s)}\omega_{\eps_0}(t-s)\dt\ds \\
    &\qquad\qquad+ \int_0^T\int_0^1\int_0^1 |x-y|\cdot \norm{Linfty01}{v(\cdot,s)}\omega_{\eps}(x-y)\dx\dy\ds\\
    &\leq \int_0^T \norm{L101}{u(\cdot,t)-v(\cdot,t)}\dt  + C\int_0^T\int_0^T |t-s|\omega_{\eps_0}(t-s)\dt\ds + \eps\norm{Linfty0T01}{v}\\
    &\leq \int_0^T \norm{L101}{u(\cdot,t)-v(\cdot,t)}\dt +C_T(\eps_0+\eps)\\
    &\to \int_0^T \norm{L101}{u(\cdot,t)-v(\cdot,t)}\dt
  \end{align*}
as $\eps,\eps_0\to 0$. Consequently, $u$ and $v$ satisfy
  \begin{equation*}
    \int_0^T\int_0^1 \Big( |u-v|\psi_t + q(u,v)\psi_x \Big)\dx\dt + \gamma \int_0^T\norm{L101}{u(\cdot,t)-v(\cdot,t)}\dt \geq 0.
  \end{equation*}
  Let now
  \begin{equation*}
    \chi_{\delta,a}(\xi)=\int_0^{\xi} \left(\omega_{\delta/2}(\zeta-\delta/2) - \omega_{\delta/2}(\zeta-(a-\delta/2))\right)\operatorname{d\zeta}
  \end{equation*}
  which is a smooth approximation to $\chi_{[0,a]}$. Then we define $\psi(x,t) = \chi_{\delta,1}(x)\chi_{\delta,T}(t)$. Taking $\delta\to 0$, we get
  \begin{multline}
    \int_0^1|u_0(x)-v_0(x)|\dx -\int_0^1|u(x,T)-v(x,T)|\dx + \gamma \int_0^T\norm{L101}{u(\cdot,t)-v(\cdot,t)}\dt \\ \geq \int_0^T q(u^{\tau}_1(t),v^{\tau}_1(t))\dt - \int_0^T q(u^{\tau}_0(t),v^{\tau}_0(t))\dt. \label{almost semigroup property}
  \end{multline}
  Note that by choosing
    \begin{equation*}
      k(t) = \begin{cases}
        u^{\tau}_0(t) &\text{if }u^{\tau}_0(t)\in I[\alpha(t),v^{\tau}_0(t)]\\
        \alpha(t) &\text{if }\alpha(t)\in I[v^{\tau}_0(t),u^{\tau}_0(t)]\\
        v^{\tau}_0(t) &\text{if }v^{\tau}_0(t)\in I[u^{\tau}_0(t),\alpha(t)]
      \end{cases}
    \end{equation*}
    in the boundary entropy condition \eqref{entropy bc} we get
    \begin{align*}
      q(u^{\tau}_0(t),v^{\tau}_0(t)) &\leq \frac{1}{2} \left( q(u_0^{\tau}(t)) - q(\alpha(t)) -\eta'(\alpha(t))(f(u_0^{\tau}(t))-f(\alpha(t)))\right.\\
      &\phantom{\leq} \phantom{\frac{1}{2}}\left.+ q(v_0^{\tau}(t)) - q(\alpha(t)) -\eta'(\alpha(t))(f(v_0^{\tau}(t))-f(\alpha(t))) \right)\leq 0
    \end{align*}
    and similarly $q(u^{\tau}_1(t),v^{\tau}_1(t)) \geq 0$ for a.e. $t$. Thus the right-hand side of \eqref{almost semigroup property} is nonnegative. An application of Gronwall's lemma finishes the proof.
\end{proof}

\section{Numerical experiments}\label{sec: numerical experiments}
In this section we want to conduct two numerical experiments to illustrate our results. Here, we choose $f(u)=u^2/2$ and $\gamma=1$. Our first numerical experiment uses a well-studied travelling wave solution of the Ostrovsky--Hunter equation with initial datum given by the `corner wave':
\begin{equation*}
  u_0(x) = \begin{cases}
    \frac{1}{6}(x-\frac{1}{2})^2+\frac{1}{6}(x-\frac{1}{2})+\frac{1}{36}, &\text{if }x\in[0,\frac{1}{2}],\\
    \frac{1}{6}(x-\frac{1}{2})^2-\frac{1}{6}(x-\frac{1}{2})+\frac{1}{36}, &\text{if }x\in[\frac{1}{2},1].
  \end{cases}
 \end{equation*}
The `corner wave' consists of two parabolas forming a sharp corner at $x=\frac{1}{2}$ (cf. Figure \ref{fig: initial datum}).
\begin{figure}[h]
	\floatbox[{\capbeside\thisfloatsetup{capbesideposition={left,top},capbesidewidth=4cm}}]{figure}[\FBwidth]
	{\caption{Initial datum for both numerical experiments.}\label{fig: initial datum}}
	% {\includegraphics[width=.6\textwidth]{Initial_datum.eps}}
	{\begin{tikzpicture}
	\begin{axis}[xtick={0,0.5,1},scaled y ticks = false, y tick label style={/pgf/number format/fixed},ytick={-0.02,0,0.03}]
		\addplot+[blue,mark=none, thick]
		coordinates{
		( 0.0 , -0.0138882785373 )
( 0.0078125 , -0.0138788689507 )
( 0.015625 , -0.0138493686252 )
( 0.0234375 , -0.0137995232476 )
( 0.03125 , -0.0137293328179 )
( 0.0390625 , -0.0136387973362 )
( 0.046875 , -0.0135279168023 )
( 0.0546875 , -0.0133966912164 )
( 0.0625 , -0.0132451205783 )
( 0.0703125 , -0.0130732048882 )
( 0.078125 , -0.0128809441461 )
( 0.0859375 , -0.0126683383518 )
( 0.09375 , -0.0124353875054 )
( 0.1015625 , -0.012182091607 )
( 0.109375 , -0.0119084506565 )
( 0.1171875 , -0.0116144646539 )
( 0.125 , -0.0113001335992 )
( 0.1328125 , -0.0109654574924 )
( 0.140625 , -0.0106104363336 )
( 0.1484375 , -0.0102350701226 )
( 0.15625 , -0.00983935885959 )
( 0.1640625 , -0.00942330254449 )
( 0.171875 , -0.0089869011773 )
( 0.1796875 , -0.00853015475803 )
( 0.1875 , -0.00805306328668 )
( 0.1953125 , -0.00755562676324 )
( 0.203125 , -0.00703784518772 )
( 0.2109375 , -0.00649971856011 )
( 0.21875 , -0.00594124688043 )
( 0.2265625 , -0.00536243014865 )
( 0.234375 , -0.0047632683648 )
( 0.2421875 , -0.00414376152886 )
( 0.25 , -0.00350390964084 )
( 0.2578125 , -0.00284371270074 )
( 0.265625 , -0.00216317070855 )
( 0.2734375 , -0.00146228366428 )
( 0.28125 , -0.000741051567925 )
( 0.2890625 , 5.25580512153e-07 )
( 0.296875 , 0.000762447781033 )
( 0.3046875 , 0.00154471503364 )
( 0.3125 , 0.00234732733832 )
( 0.3203125 , 0.0031702846951 )
( 0.328125 , 0.00401358710395 )
( 0.3359375 , 0.00487723456489 )
( 0.34375 , 0.00576122707791 )
( 0.3515625 , 0.00666556464301 )
( 0.359375 , 0.0075902472602 )
( 0.3671875 , 0.00853527492947 )
( 0.375 , 0.00950064765082 )
( 0.3828125 , 0.0104863654243 )
( 0.390625 , 0.0114924282498 )
( 0.3984375 , 0.0125188361274 )
( 0.40625 , 0.0135655890571 )
( 0.4140625 , 0.0146326870388 )
( 0.421875 , 0.0157201300727 )
( 0.4296875 , 0.0168279181586 )
( 0.4375 , 0.0179560512967 )
( 0.4453125 , 0.0191045294868 )
( 0.453125 , 0.0202733527289 )
( 0.4609375 , 0.0214625210232 )
( 0.46875 , 0.0226720343696 )
( 0.4765625 , 0.023901892768 )
( 0.484375 , 0.0251520962185 )
( 0.4921875 , 0.0264226447211 )
( 0.5 , 0.0274531216092 )
( 0.5078125 , 0.0265508185493 )
( 0.515625 , 0.0252782355414 )
( 0.5234375 , 0.0240259975857 )
( 0.53125 , 0.0227941046821 )
( 0.5390625 , 0.0215825568305 )
( 0.546875 , 0.020391354031 )
( 0.5546875 , 0.0192204962836 )
( 0.5625 , 0.0180699835883 )
( 0.5703125 , 0.0169398159451 )
( 0.578125 , 0.0158299933539 )
( 0.5859375 , 0.0147405158149 )
( 0.59375 , 0.0136713833279 )
( 0.6015625 , 0.012622595893 )
( 0.609375 , 0.0115941535102 )
( 0.6171875 , 0.0105860561795 )
( 0.625 , 0.00959830390082 )
( 0.6328125 , 0.00863089667426 )
( 0.640625 , 0.00768383449978 )
( 0.6484375 , 0.00675711737739 )
( 0.65625 , 0.00585074530707 )
( 0.6640625 , 0.00496471828885 )
( 0.671875 , 0.0040990363227 )
( 0.6796875 , 0.00325369940864 )
( 0.6875 , 0.00242870754666 )
( 0.6953125 , 0.00162406073676 )
( 0.703125 , 0.00083975897895 )
( 0.7109375 , 7.58022732205e-05 )
( 0.71875 , -0.000667809380425 )
( 0.7265625 , -0.00139107598199 )
( 0.734375 , -0.00209399753147 )
( 0.7421875 , -0.00277657402886 )
( 0.75 , -0.00343880547418 )
( 0.7578125 , -0.0040806918674 )
( 0.765625 , -0.00470223320855 )
( 0.7734375 , -0.00530342949761 )
( 0.78125 , -0.00588428073459 )
( 0.7890625 , -0.00644478691949 )
( 0.796875 , -0.0069849480523 )
( 0.8046875 , -0.00750476413303 )
( 0.8125 , -0.00800423516168 )
( 0.8203125 , -0.00848336113824 )
( 0.828125 , -0.00894214206272 )
( 0.8359375 , -0.00938057793511 )
( 0.84375 , -0.00979866875543 )
( 0.8515625 , -0.0101964145237 )
( 0.859375 , -0.0105738152398 )
( 0.8671875 , -0.0109308709039 )
( 0.875 , -0.0112675815158 )
( 0.8828125 , -0.0115839470757 )
( 0.890625 , -0.0118799675836 )
( 0.8984375 , -0.0121556430393 )
( 0.90625 , -0.0124109734429 )
( 0.9140625 , -0.0126459587945 )
( 0.921875 , -0.012860599094 )
( 0.9296875 , -0.0130548943414 )
( 0.9375 , -0.0132288445367 )
( 0.9453125 , -0.0133824496799 )
( 0.953125 , -0.0135157097711 )
( 0.9609375 , -0.0136286248101 )
( 0.96875 , -0.0137211947971 )
( 0.9765625 , -0.013793419732 )
( 0.984375 , -0.0138452996148 )
( 0.9921875 , -0.0138768344455 )
( 1.0 , -0.013887769911 )
		};
	\end{axis}
	\end{tikzpicture}}
\end{figure}
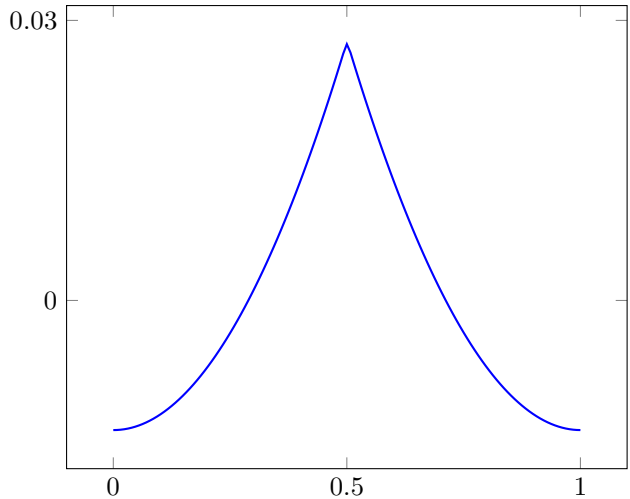
The travelling wave solution is
\begin{equation*}
  u_{\operatorname{ex}}(x,t) = u_0\left(x-\frac{t}{36} - \left\lfloor x-\frac{t}{36}\right\rfloor\right)
\end{equation*}
which returns to its initial state after a period of $T=36$. The `corner wave' is the limit case of a family of smooth travelling wave solutions that has been investigated by several authors \cite{hunter1990numerical,ostrovsky1978nonlinear,boyd2005ostrovsky,stepanyants2006stationary,parkes2007explicit}. In this section we will not consider $P(0)=0$, but $\int_0^1 P =0$, which gives
\begin{equation*}
  (P[u])(x,t) = \int_0^x u(y,t)\dy - \int_0^1\int_0^y u(z,t)\dz.
 \end{equation*}
This is motivated by the fact that the latter choice limits the growth of the $\m{Linfty}$ norm of the solution for our experiments.
Figure \ref{fig: LxF correct boundaries} shows the exact entropy solution and a numerical solution both at $T=36$.
\begin{figure}[h]
  	\floatbox[{\capbeside\thisfloatsetup{capbesideposition={left,top},capbesidewidth=4cm}}]{figure}[\FBwidth]
  	{\caption{Explicit and numerical solution for Experiment $1$ at $T=36$}\label{fig: LxF correct boundaries}}
    % {\includegraphics[width=.6\textwidth]{LxF_periodic.eps}}
    {\begin{tikzpicture}
	\begin{axis}[xtick={0,0.5,1},scaled y ticks = false, y tick label style={/pgf/number format/fixed},ytick={-0.02,0,0.03}]
		\addplot+[blue,mark=none, thick]
		coordinates{
		( 0.0 , -0.0138882785373 )
( 0.0078125 , -0.0138788689507 )
( 0.015625 , -0.0138493686252 )
( 0.0234375 , -0.0137995232476 )
( 0.03125 , -0.0137293328179 )
( 0.0390625 , -0.0136387973362 )
( 0.046875 , -0.0135279168023 )
( 0.0546875 , -0.0133966912164 )
( 0.0625 , -0.0132451205783 )
( 0.0703125 , -0.0130732048882 )
( 0.078125 , -0.0128809441461 )
( 0.0859375 , -0.0126683383518 )
( 0.09375 , -0.0124353875054 )
( 0.1015625 , -0.012182091607 )
( 0.109375 , -0.0119084506565 )
( 0.1171875 , -0.0116144646539 )
( 0.125 , -0.0113001335992 )
( 0.1328125 , -0.0109654574924 )
( 0.140625 , -0.0106104363336 )
( 0.1484375 , -0.0102350701226 )
( 0.15625 , -0.00983935885959 )
( 0.1640625 , -0.00942330254449 )
( 0.171875 , -0.0089869011773 )
( 0.1796875 , -0.00853015475803 )
( 0.1875 , -0.00805306328668 )
( 0.1953125 , -0.00755562676324 )
( 0.203125 , -0.00703784518772 )
( 0.2109375 , -0.00649971856011 )
( 0.21875 , -0.00594124688043 )
( 0.2265625 , -0.00536243014865 )
( 0.234375 , -0.0047632683648 )
( 0.2421875 , -0.00414376152886 )
( 0.25 , -0.00350390964084 )
( 0.2578125 , -0.00284371270074 )
( 0.265625 , -0.00216317070855 )
( 0.2734375 , -0.00146228366428 )
( 0.28125 , -0.000741051567925 )
( 0.2890625 , 5.25580512153e-07 )
( 0.296875 , 0.000762447781033 )
( 0.3046875 , 0.00154471503364 )
( 0.3125 , 0.00234732733832 )
( 0.3203125 , 0.0031702846951 )
( 0.328125 , 0.00401358710395 )
( 0.3359375 , 0.00487723456489 )
( 0.34375 , 0.00576122707791 )
( 0.3515625 , 0.00666556464301 )
( 0.359375 , 0.0075902472602 )
( 0.3671875 , 0.00853527492947 )
( 0.375 , 0.00950064765082 )
( 0.3828125 , 0.0104863654243 )
( 0.390625 , 0.0114924282498 )
( 0.3984375 , 0.0125188361274 )
( 0.40625 , 0.0135655890571 )
( 0.4140625 , 0.0146326870388 )
( 0.421875 , 0.0157201300727 )
( 0.4296875 , 0.0168279181586 )
( 0.4375 , 0.0179560512967 )
( 0.4453125 , 0.0191045294868 )
( 0.453125 , 0.0202733527289 )
( 0.4609375 , 0.0214625210232 )
( 0.46875 , 0.0226720343696 )
( 0.4765625 , 0.023901892768 )
( 0.484375 , 0.0251520962185 )
( 0.4921875 , 0.0264226447211 )
( 0.5 , 0.0274531216092 )
( 0.5078125 , 0.0265508185493 )
( 0.515625 , 0.0252782355414 )
( 0.5234375 , 0.0240259975857 )
( 0.53125 , 0.0227941046821 )
( 0.5390625 , 0.0215825568305 )
( 0.546875 , 0.020391354031 )
( 0.5546875 , 0.0192204962836 )
( 0.5625 , 0.0180699835883 )
( 0.5703125 , 0.0169398159451 )
( 0.578125 , 0.0158299933539 )
( 0.5859375 , 0.0147405158149 )
( 0.59375 , 0.0136713833279 )
( 0.6015625 , 0.012622595893 )
( 0.609375 , 0.0115941535102 )
( 0.6171875 , 0.0105860561795 )
( 0.625 , 0.00959830390082 )
( 0.6328125 , 0.00863089667426 )
( 0.640625 , 0.00768383449978 )
( 0.6484375 , 0.00675711737739 )
( 0.65625 , 0.00585074530707 )
( 0.6640625 , 0.00496471828885 )
( 0.671875 , 0.0040990363227 )
( 0.6796875 , 0.00325369940864 )
( 0.6875 , 0.00242870754666 )
( 0.6953125 , 0.00162406073676 )
( 0.703125 , 0.00083975897895 )
( 0.7109375 , 7.58022732205e-05 )
( 0.71875 , -0.000667809380425 )
( 0.7265625 , -0.00139107598199 )
( 0.734375 , -0.00209399753147 )
( 0.7421875 , -0.00277657402886 )
( 0.75 , -0.00343880547418 )
( 0.7578125 , -0.0040806918674 )
( 0.765625 , -0.00470223320855 )
( 0.7734375 , -0.00530342949761 )
( 0.78125 , -0.00588428073459 )
( 0.7890625 , -0.00644478691949 )
( 0.796875 , -0.0069849480523 )
( 0.8046875 , -0.00750476413303 )
( 0.8125 , -0.00800423516168 )
( 0.8203125 , -0.00848336113824 )
( 0.828125 , -0.00894214206272 )
( 0.8359375 , -0.00938057793511 )
( 0.84375 , -0.00979866875543 )
( 0.8515625 , -0.0101964145237 )
( 0.859375 , -0.0105738152398 )
( 0.8671875 , -0.0109308709039 )
( 0.875 , -0.0112675815158 )
( 0.8828125 , -0.0115839470757 )
( 0.890625 , -0.0118799675836 )
( 0.8984375 , -0.0121556430393 )
( 0.90625 , -0.0124109734429 )
( 0.9140625 , -0.0126459587945 )
( 0.921875 , -0.012860599094 )
( 0.9296875 , -0.0130548943414 )
( 0.9375 , -0.0132288445367 )
( 0.9453125 , -0.0133824496799 )
( 0.953125 , -0.0135157097711 )
( 0.9609375 , -0.0136286248101 )
( 0.96875 , -0.0137211947971 )
( 0.9765625 , -0.013793419732 )
( 0.984375 , -0.0138452996148 )
( 0.9921875 , -0.0138768344455 )
( 1.0 , -0.013887769911 )
		};
		\addplot[mark=none, thick, dashed, red]
		coordinates{
		( 0.0 , -0.0138869936069 )
( 0.0078125 , -0.0133645417875 )
( 0.015625 , -0.0130613106679 )
( 0.0234375 , -0.0128532955302 )
( 0.03125 , -0.0126835336338 )
( 0.0390625 , -0.0125223630693 )
( 0.046875 , -0.0123573487692 )
( 0.0546875 , -0.0121797677976 )
( 0.0625 , -0.0119878481835 )
( 0.0703125 , -0.011777795171 )
( 0.078125 , -0.0115506684791 )
( 0.0859375 , -0.0113037822998 )
( 0.09375 , -0.0110391159006 )
( 0.1015625 , -0.0107541391208 )
( 0.109375 , -0.0104512584257 )
( 0.1171875 , -0.0101278477257 )
( 0.125 , -0.00978660272618 )
( 0.1328125 , -0.00942473134206 )
( 0.140625 , -0.00904517420698 )
( 0.1484375 , -0.00864495457025 )
( 0.15625 , -0.00822724031872 )
( 0.1640625 , -0.00778886815592 )
( 0.171875 , -0.00733322378619 )
( 0.1796875 , -0.00685696230258 )
( 0.1875 , -0.00636368063886 )
( 0.1953125 , -0.00584986092414 )
( 0.203125 , -0.00531930599913 )
( 0.2109375 , -0.00476833643882 )
( 0.21875 , -0.00420095662822 )
( 0.2265625 , -0.00361333964863 )
( 0.234375 , -0.00300968826599 )
( 0.2421875 , -0.00238604514392 )
( 0.25 , -0.00174681006279 )
( 0.2578125 , -0.00108791617059 )
( 0.265625 , -0.000413961668201 )
( 0.2734375 , 0.000279203693621 )
( 0.28125 , 0.000986777612747 )
( 0.2890625 , 0.00171296013969 )
( 0.296875 , 0.00245273259435 )
( 0.3046875 , 0.00321030010297 )
( 0.3125 , 0.00398040508554 )
( 0.3203125 , 0.00476719594471 )
( 0.328125 , 0.00556513758549 )
( 0.3359375 , 0.00637823301272 )
( 0.34375 , 0.00720060538501 )
( 0.3515625 , 0.00803598038187 )
( 0.359375 , 0.00887803390737 )
( 0.3671875 , 0.00973001150672 )
( 0.375 , 0.0105849693642 )
( 0.3828125 , 0.0114453526415 )
( 0.390625 , 0.0123033174703 )
( 0.3984375 , 0.0131599943574 )
( 0.40625 , 0.0140061917615 )
( 0.4140625 , 0.0148409008472 )
( 0.421875 , 0.0156529042684 )
( 0.4296875 , 0.016437780429 )
( 0.4375 , 0.0171814013874 )
( 0.4453125 , 0.0178741909916 )
( 0.453125 , 0.0184984588597 )
( 0.4609375 , 0.0190378575276 )
( 0.46875 , 0.0194723766847 )
( 0.4765625 , 0.0197795858466 )
( 0.484375 , 0.0199438946974 )
( 0.4921875 , 0.0199432759265 )
( 0.5 , 0.019781007942 )
( 0.5078125 , 0.0194458985715 )
( 0.515625 , 0.0189726590327 )
( 0.5234375 , 0.0183609849876 )
( 0.53125 , 0.0176702335088 )
( 0.5390625 , 0.0168941559081 )
( 0.546875 , 0.0160970083932 )
( 0.5546875 , 0.0152533904209 )
( 0.5625 , 0.014422808197 )
( 0.5703125 , 0.0135636776494 )
( 0.578125 , 0.0127322861197 )
( 0.5859375 , 0.0118794163406 )
( 0.59375 , 0.0110602560516 )
( 0.6015625 , 0.0102235526328 )
( 0.609375 , 0.00942325245173 )
( 0.6171875 , 0.0086086982782 )
( 0.625 , 0.00783172970539 )
( 0.6328125 , 0.00704354063942 )
( 0.640625 , 0.00629323224575 )
( 0.6484375 , 0.00553447272098 )
( 0.65625 , 0.00481334621021 )
( 0.6640625 , 0.00408625627669 )
( 0.671875 , 0.00339623654992 )
( 0.6796875 , 0.00270246590251 )
( 0.6875 , 0.00204503148396 )
( 0.6953125 , 0.00138579827789 )
( 0.703125 , 0.000762073354944 )
( 0.7109375 , 0.000138247497177 )
( 0.71875 , -0.000450974269605 )
( 0.7265625 , -0.00103887338466 )
( 0.734375 , -0.00159319290044 )
( 0.7421875 , -0.00214510352427 )
( 0.75 , -0.00266469158242 )
( 0.7578125 , -0.00318123889044 )
( 0.765625 , -0.00366713311086 )
( 0.7734375 , -0.00414997380304 )
( 0.78125 , -0.00460446899714 )
( 0.7890625 , -0.00505669297115 )
( 0.796875 , -0.00548373174626 )
( 0.8046875 , -0.00591018799603 )
( 0.8125 , -0.00631557635523 )
( 0.8203125 , -0.00672290813423 )
( 0.828125 , -0.00711410415004 )
( 0.8359375 , -0.00751022065885 )
( 0.84375 , -0.00789541776535 )
( 0.8515625 , -0.00828815788823 )
( 0.859375 , -0.00867449069892 )
( 0.8671875 , -0.00906944780485 )
( 0.875 , -0.00946049293696 )
( 0.8828125 , -0.00985844573332 )
( 0.890625 , -0.0102517784356 )
( 0.8984375 , -0.0106467820545 )
( 0.90625 , -0.0110328671663 )
( 0.9140625 , -0.0114122958564 )
( 0.921875 , -0.0117757989269 )
( 0.9296875 , -0.0121229013021 )
( 0.9375 , -0.0124462652985 )
( 0.9453125 , -0.0127442624037 )
( 0.953125 , -0.0130118786702 )
( 0.9609375 , -0.0132475617796 )
( 0.96875 , -0.0134484943935 )
( 0.9765625 , -0.0136137467577 )
( 0.984375 , -0.0137421256137 )
( 0.9921875 , -0.0138333006876 )
( 1.0 , -0.0138869936069 )
};
	\end{axis}
	\end{tikzpicture}}
\end{figure}
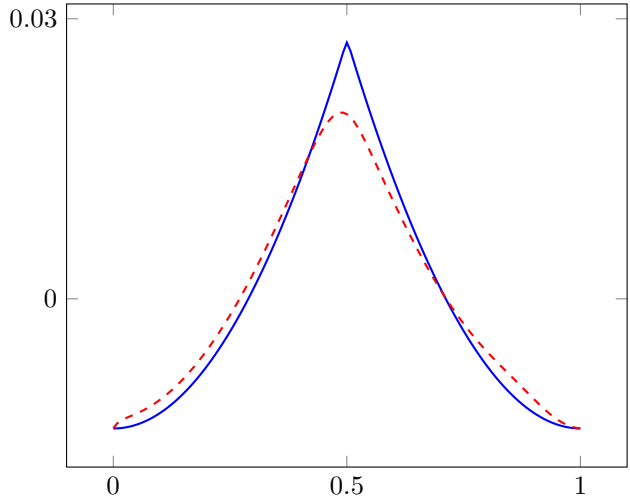
The numerical solution is calculated by the Lax-Friedrichs method with boundary conditions set as the explicit solution at $x=0$ and $x=1$ respectively and a grid discretization parameter of $\Delta x=2^{-7}$.

For this and all subsequent numerical experiments we use\footnote{Here, we have $\norm{Linfty01}{f'(u_0)}=1/36$ and therefore $\lambda=\Delta t/\Delta x$ should satisfy $\lambda\leq 36$. However, since the $\m{Linfty}$ bound from Lemma \ref{Linfty bound} allows for some growth of $\norm*{\infty}{u^n}$ choosing a smaller $\lambda$ can be neccessary.} $\Delta t/\Delta x=25$. Additionally, for the first experiment the known exact entropy solution is used to calculate the error:
\begin{equation*}
	\operatorname{err}_{\m{L1}}^1(\Delta t)=\norm{L101}{u_{\Delta t}(\cdot,36)-u_{\operatorname{ex}}(\cdot,36)}.
\end{equation*}
Table \ref{tbl: LxF correct boundaries} shows the $\m{L1}$ error between various numerical solutions and the exact solution, as well as the respective experimental convergence rates.
\begin{table}[h]
\renewcommand{\arraystretch}{1.3}
\floatbox[{\capbeside\thisfloatsetup{capbesideposition={left,top},capbesidewidth=4cm}}]{table}[\FBwidth]
{\caption{$\m{L1}$ errors and convergence rates for Experiment 1}\label{tbl: LxF correct boundaries}}
{\begin{tabular*}{.6\textwidth}{lllll}
  \toprule
  $\Delta x$ & Lax-Friedrichs & Rate & Engquist-Osher & Rate\\
  \midrule
  $2^{-6}$ & $2.84\cdot 10^{-3}$ && $1.39\cdot10^{-3}$ & \\
  $2^{-7}$ & $ 1.72\cdot 10^{-3}$ &  $0.72$ & $6.92\cdot 10^{-4}$ & $1.00$ \\ 
  $2^{-8}$ & $ 9.71\cdot 10^{-4}$ &  $0.82$ & $3.61\cdot 10^{-4}$ & $0.94$ \\ 
  $2^{-9}$ & $ 5.32\cdot 10^{-4}$ &  $0.86$ & $1.90\cdot 10^{-4}$ & $0.93$ \\ 
  $2^{-10}$ & $ 2.83\cdot 10^{-4}$ &  $0.91$ & $1.01\cdot 10^{-4}$ & $0.91$ \\ 
  \bottomrule
\end{tabular*}}
\end{table}
Comparing these results to Table 1 in~\cite{coclite2017convergent}, we see that our numerical scheme is consistent with the periodic case.

In our second experiment we use the same initial datum, but set the right boundary datum to zero.
Figure \ref{fig: LxF right boundary off} displays two numerical solutions, one on a moderate mesh ($\Delta x =2^{-7}$) calculated with the Lax--Friedrichs flux and one on a fine mesh ($\Delta x^*=2^{-11}$) calculated with the Engquist--Osher flux.
\begin{figure}[h]
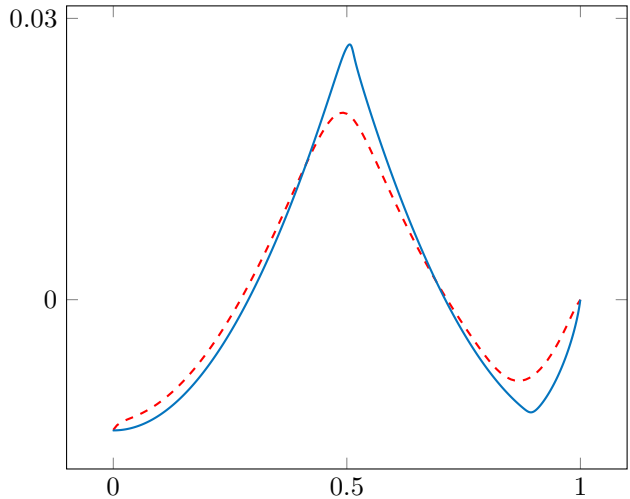

  	\floatbox[{\capbeside\thisfloatsetup{capbesideposition={left,top},capbesidewidth=4cm}}]{figure}[\FBwidth]
  	{\caption{Numerical solutions for Experiment 2 at $T=36$ calculated with the Lax--Friedrichs flux and $\Delta x=2^{-7}$ (dashed) and with the Engquist--Osher flux and $\Delta x^*=2^{-11}$ (straight)
    %Numerical solutions at $T=36$ calculated with the Lax-Friedrichs flux (dashed) and numerical solution on fine mesh calculated with the Engquist--Osher flux (straight line) ($\Delta x =2^{-7}$ respectively $\Delta x = 2^{-11}$, CFL number $25$)
    }\label{fig: LxF right boundary off}}
    % {\includegraphics[width=.6\textwidth]{Right_boundary_off.eps}}
    {
}
\end{figure}
With no explicit entropy solution at hand we consider a numerical solution on a fine grid ($\Delta x^*=2^{-11}$) in order to calculate the $\m{L1}$ errors in the second experiment, i.e.,
\begin{equation*}
	\operatorname{err}_{\m{L1}}^2(\Delta t)=\norm{L101}{u_{\Delta t}(\cdot,36)-u_{\Delta t^*}(\cdot,36)}.
\end{equation*}
Here, $u_{\Delta t}$ and $u_{\Delta t^*}$ are always calculated based on the same numerical method.
Finally, in Table \ref{tbl: LxF right boundary off} we compare the $\m{L1}$ errors between various numerical solutions and provide the experimental convergence rates.
\begin{table}[h]
\renewcommand{\arraystretch}{1.3}
\floatbox[{\capbeside\thisfloatsetup{capbesideposition={left,top},capbesidewidth=4cm}}]{table}[\FBwidth]
{\caption{$\m{L1}$ errors and convergence rates for Experiment 2}\label{tbl: LxF right boundary off}}
{\begin{tabular*}{.6\textwidth}{lllll}
  \toprule
  $\Delta x$ & Lax-Friedrichs & Rate & Engquist-Osher & Rate\\
  \midrule
  $2^{-6}$ &  $3.00\cdot 10^{-3}$  &         & $ 1.36\cdot 10^{-3}$  & \\
  $2^{-7}$ &  $1.90\cdot 10^{-3}$  &  $0.66$ & $ 6.60\cdot 10^{-4}$ & $1.04$ \\ 
  $2^{-8}$ &  $1.16\cdot 10^{-3}$  &  $0.71$ & $ 3.24\cdot 10^{-4}$ & $1.03$ \\ 
  $2^{-9}$ &  $6.88\cdot 10^{-4}$  &  $0.75$ & $ 1.50\cdot 10^{-4}$ & $1.11$ \\ 
  $2^{-10}$ & $ 4.05\cdot 10^{-4}$ &  $0.76$ & $ 5.83\cdot 10^{-5}$ & $1.36$ \\ 
  \bottomrule
\end{tabular*}}
\end{table}
One clearly sees that the Engquist--Osher flux leads to a better approximation in this experiment. This is due to the fact that the homogeneous boundary condition at $x=1$ constitutes a shock that propagates into the domain and that shocks are resolved better with the Engquist--Osher flux.

For conservation laws in $\R$ without source term the classical result concerning convergence rates in $\m{L1}$, due to Kuznetsov \cite{kuznetsov1976accuracy}, gives a convergence rate of $O(\Delta x^{1/2})$. The same convergence rate was shown in \cite{coclite2017convergent} for the Ostrovsky--Hunter equation with periodic boundary conditions. Although theoretical results estimating the convergence rate in the case of Dirichlet boundary conditions are highly desirable, such results are currently not at hand. However, in the absence of source terms Ohlberger and Vovelle \cite{ohlberger2006error} proof a rate of $O(\Delta x^{1/6})$ in a very general setting.

\section*{Acknowledgements}
The authors want to thank Nils Henrik Risebro from the University of Oslo for many insightful discussions.

% \begin{acknowledgements}
% The authors want to thank Nils Henrik Risebro from the University of Oslo for many insightful discussions.
% \end{acknowledgements}

% BibTeX users please use one of
% \bibliographystyle{spbasic}      % basic style, author-year citations
\bibliographystyle{siam}      % mathematics and physical sciences
% \bibliography{literature.bib}  % name your BibTeX data base

\begin{thebibliography}{10}

\bibitem{Amiranashvili2010}
{\sc S.~Amiranashvili, A.~G. Vladimirov, and U.~Bandelow}, {\em A model
  equation for ultrashort optical pulses around the zero dispersion frequency},
  The European Physical Journal D, 58 (2010), pp.~219--226.

\bibitem{doi:10.1080/03605307908820117}
{\sc C.~Bardos, A.~Y. Leroux, and J.~C. Nedelec}, {\em First order quasilinear
  equations with boundary conditions}, Communications in Partial Differential
  Equations, 4 (1979), pp.~1017--1034.

\bibitem{boyd2005ostrovsky}
{\sc J.~P. Boyd}, {\em {O}strovsky and {H}unter's generic wave equation for
  weakly dispersive waves: {M}atched asymptotic and pseudospectral study of the
  paraboloidal travelling waves (corner and near-corner waves)}, European
  Journal of Applied Mathematics, 16 (2005), pp.~65--81.

\bibitem{brunelli2013hamiltonian}
{\sc J.~C. Brunelli and S.~Sakovich}, {\em Hamiltonian structures for the
  {O}strovsky--{V}akhnenko equation}, Communications in Nonlinear Science and
  Numerical Simulation, 18 (2013), pp.~56--62.

\bibitem{coclite2014some}
{\sc G.~Coclite, L.~di~Ruvo, and K.~Karlsen}, {\em Some wellposedness results
  for the {O}strovsky--{H}unter equation}, in Hyperbolic conservation laws and
  related analysis with applications, Springer, 2014, pp.~143--159.

\bibitem{COCLITE20093823}
{\sc G.~Coclite, K.~Karlsen, and Y.-S. Kwon}, {\em Initial--boundary value
  problems for conservation laws with source terms and the degasperis--procesi
  equation}, Journal of Functional Analysis, 257 (2009), pp.~3823 -- 3857.

\bibitem{coclite2017convergent}
{\sc G.~Coclite, J.~Ridder, and N.~Risebro}, {\em A convergent finite
  difference scheme for the {O}strovsky-{H}unter equation on a bounded domain},
  BIT Numerical Mathematics, 57 (2017), pp.~93--122.

\bibitem{coclite2015oleinik}
{\sc G.~M. Coclite and L.~di~Ruvo}, {\em Oleinik type estimates for the
  {O}strovsky--{H}unter equation}, Journal of Mathematical Analysis and
  Applications, 423 (2015), pp.~162--190.

\bibitem{coclite2015well}
\leavevmode\vrule height 2pt depth -1.6pt width 23pt, {\em Well-posedness of
  bounded solutions of the non-homogeneous initial-boundary value problem for
  the {O}strovsky--{H}unter equation}, Journal of Hyperbolic Differential
  Equations, 12 (2015), pp.~221--248.

\bibitem{coclite2015short}
\leavevmode\vrule height 2pt depth -1.6pt width 23pt, {\em Well-posedness
  results for the short pulse equation}, Zeitschrift f{\"u}r angewandte
  Mathematik und Physik, 66 (2015), pp.~1529--1557.

\bibitem{coclite2015wellposedness}
\leavevmode\vrule height 2pt depth -1.6pt width 23pt, {\em Wellposedness of
  bounded solutions of the non-homogeneous initial boundary for the short pulse
  equation}, Bollettino dell'Unione Matematica Italiana, 8 (2015), pp.~31--44.

\bibitem{coclite2016well}
\leavevmode\vrule height 2pt depth -1.6pt width 23pt, {\em Well-posedness of
  the {O}strovsky--{H}unter equation under the combined effects of dissipation
  and short-wave dispersion}, Journal of Evolution Equations, 16 (2016),
  pp.~365--389.

\bibitem{coclite2018initial}
{\sc G.~M. Coclite, L.~di~Ruvo, and K.~H. Karlsen}, {\em The
  initial-boundary-value problem for an {O}strovsky--{H}unter type equation},
  Non-Linear Partial Differential Equations, Mathematical Physics, and
  Stochastic Analysis,  (2018), pp.~97--109.

\bibitem{BOUTETDEMONVEL2014189}
{\sc A.~B. de~Monvel and D.~Shepelsky}, {\em The {O}strovsky--{V}akhnenko
  equation: {A} {R}iemann--{H}ilbert approach}, Comptes Rendus Mathematique,
  352 (2014), pp.~189 -- 195.

\bibitem{di2013discontinuous}
{\sc L.~di~Ruvo}, {\em Discontinuous solutions for the {O}strovsky-{H}unter
  equation and two-phase flows}, PhD thesis, University of Bari, 2013.

\bibitem{Dubois1989}
{\sc F.~Dubois and P.~Le~Floch}, {\em Boundary Conditions for Nonlinear
  Hyperbolic Systems of Conservation Laws}, Vieweg+Teubner Verlag, Wiesbaden,
  1989, pp.~96--104.

\bibitem{grimshaw2012reduced}
{\sc R.~H. Grimshaw, K.~Helfrich, and E.~R. Johnson}, {\em The reduced
  {O}strovsky equation: integrability and breaking}, Studies in Applied
  Mathematics, 129 (2012), pp.~414--436.

\bibitem{harten1983high}
{\sc A.~Harten}, {\em High resolution schemes for hyperbolic conservation
  laws}, Journal of {C}omputational {P}hysics, 49 (1983), pp.~357--393.

\bibitem{holden2015front}
{\sc H.~Holden and N.~H. Risebro}, {\em Front tracking for hyperbolic
  conservation laws}, vol.~152, Springer, 2015.

\bibitem{hunter1990numerical}
{\sc J.~K. Hunter}, {\em Numerical solutions of some nonlinear dispersive wave
  equations}, Lect. Appl. Math, 26 (1990), pp.~301--316.

\bibitem{karlsen2014L1}
{\sc K.~Karlsen, N.~Risebro, and E.~Storr{\o}sten}, {\em {$L^1$} error
  estimates for difference approximations of degenerate convection-diffusion
  equations}, Mathematics of Computation, 83 (2014), pp.~2717--2762.

\bibitem{kruvzkov1970first}
{\sc S.~N. Kru{\v{z}}kov}, {\em First order quasilinear equations in several
  independent variables}, Mathematics of the USSR-Sbornik, 10 (1970), p.~217.

\bibitem{kuznetsov1976accuracy}
{\sc N.~Kuznetsov}, {\em Accuracy of some approximate methods for computing the
  weak solutions of a first-order quasi-linear equation}, USSR Computational
  Mathematics and Mathematical Physics, 16 (1976), pp.~105--119.

\bibitem{liu2009wave}
{\sc Y.~Liu, D.~Pelinovsky, and A.~Sakovich}, {\em Wave breaking in the
  short-pulse equation}, Dynamics of PDE, 6 (2009), pp.~291--310.

\bibitem{liu2010wave}
\leavevmode\vrule height 2pt depth -1.6pt width 23pt, {\em Wave breaking in the
  {O}strovsky--{H}unter equation}, SIAM Journal on Mathematical Analysis, 42
  (2010), pp.~1967--1985.

\bibitem{morrison1999n}
{\sc A.~Morrison, E.~Parkes, and V.~Vakhnenko}, {\em The {N} loop soliton
  solution of the {V}akhnenko equation}, Nonlinearity, 12 (1999), p.~1427.

\bibitem{ohlberger2006error}
{\sc M.~Ohlberger and J.~Vovelle}, {\em Error estimate for the approximation of
  nonlinear conservation laws on bounded domains by the finite volume method},
  Mathematics of Computation, 75 (2006), pp.~113--150.

\bibitem{ostrovsky1978nonlinear}
{\sc L.~Ostrovsky}, {\em Nonlinear internal waves in a rotating ocean},
  Oceanology, 18 (1978), pp.~119--125.

\bibitem{parkes1993stability}
{\sc E.~Parkes}, {\em The stability of solutions of {V}akhnenko's equation},
  Journal of Physics A: Mathematical and General, 26 (1993), p.~6469.

\bibitem{parkes2007explicit}
\leavevmode\vrule height 2pt depth -1.6pt width 23pt, {\em Explicit solutions
  of the reduced {O}strovsky equation}, Chaos, Solitons \& Fractals, 31 (2007),
  pp.~602--610.

\bibitem{s2006weak}
{\sc I.~S~Strub and A.~M~Bayen}, {\em Weak formulation of boundary conditions
  for scalar conservation laws: An application to highway traffic modelling},
  International Journal of Robust and Nonlinear Control, 16 (2006),
  pp.~733--748.

\bibitem{schafer2004propagation}
{\sc T.~Sch{\"a}fer and C.~Wayne}, {\em Propagation of ultra-short optical
  pulses in cubic nonlinear media}, Physica D: Nonlinear Phenomena, 196 (2004),
  pp.~90--105.

\bibitem{stepanyants2006stationary}
{\sc Y.~A. Stepanyants}, {\em On stationary solutions of the reduced
  {O}strovsky equation: Periodic waves, compactons and compound solitons},
  Chaos, Solitons \& Fractals, 28 (2006), pp.~193--204.

\bibitem{vakhnenko1992solitons}
{\sc V.~Vakhnenko}, {\em Solitons in a nonlinear model medium}, Journal of
  Physics A: Mathematical and General, 25 (1992), p.~4181.

\bibitem{vakhnenko1998two}
{\sc V.~Vakhnenko and E.~Parkes}, {\em The two loop soliton solution of the
  {V}akhnenko equation}, Nonlinearity, 11 (1998), p.~1457.

\bibitem{vakhnenko2002calculation}
\leavevmode\vrule height 2pt depth -1.6pt width 23pt, {\em The calculation of
  multi-soliton solutions of the {V}akhnenko equation by the inverse scattering
  method}, Chaos, Solitons \& Fractals, 13 (2002), pp.~1819--1826.

\end{thebibliography}

\end{document}